%% file: shanks_bias.tex
\documentclass[letterpaper, 10pt]{article}

\input{preamble}

\begin{document}


\title{Shanks' bias in function fields}


\author{Seewoo Lee}
\date{}


\maketitle


\begin{abstract}
    We study the function field analogue of Shanks' bias. For Liouville function \(\lambda(f)\), we compare the number of monic polynomials $f$ with $\lambda(f) \chi_m(f) = 1$ and $\lambda(f) \chi_m(f) = -1$ for a nontrivial quadratic character \( \chi_m \) modulo a monic square-free polynomial $m$ over a finite field.
    Under the Grand Simplicity Hypothesis (GSH) for \(L\)-functions, we prove that \(\lambda \cdot \chi_m\) is biased towards \(+1\).
    We also give some examples where GSH is violated.
\end{abstract}

\input{src/1intro}
\input{src/2preliminary}
\input{src/3lfunc}
\input{src/4density}
\input{src/5examples}


\bibliographystyle{acm} 
\bibliography{refs} 

{Department of Mathematics, University of California---Berkeley, Berkeley, CA 94720, USA}

\href{mailto:seewoo5@berkeley.edu}{seewoo5@berkeley.edu}


\end{document}

%% file: preamble.tex
\usepackage[dvipsnames]{xcolor}

\usepackage{
  amsmath, amsthm, amssymb, mathtools, dsfont,          
  graphicx, wrapfig, subcaption, float,                            
  listings, color, inconsolata, pythonhighlight,               
  fancyhdr, sectsty, hyperref, enumerate, enumitem, framed }   

\usepackage{hyperref, lipsum} 
\usepackage{siunitx}

\usepackage{cleveref}
\usepackage{multicol}
\usepackage{multirow}
\usepackage{makecell}

\usepackage{newpxtext, newpxmath, inconsolata}
\usepackage{amsfonts}
\usepackage{pgfplots}
\pgfplotsset{compat=1.12}
\usepackage{tkz-fct}
\usepackage{tikz,forest}
\usetikzlibrary{patterns}
\usetikzlibrary{arrows.meta}
\useforestlibrary{edges}

\tikzset{
decision/.style={diamond, inner sep=1pt},
chance/.style={circle, minimum width=10pt, draw=blue, fill=none, thick, inner sep=0pt},
}

\usepackage{tikz-cd}
\usepackage{enumitem}
\usepackage[title]{appendix}
\usepackage{booktabs}
\usepackage{dynkin-diagrams}

\usepackage{geometry}

\usepackage[bottom]{footmisc}
\usepackage{mathtools}

\allowdisplaybreaks

\usepackage[font={it,footnotesize}]{caption}

\hypersetup{colorlinks=true, linkcolor=RoyalBlue, citecolor=RedOrange, urlcolor=ForestGreen}

\usepackage{titlesec}
\titleformat{\section}{\large\bfseries}{\thesection\;\;\;}{0em}{}
\titleformat{\subsection}{\normalsize\bfseries\selectfont}{\thesubsection\;\;\;}{0em}{}

\setlist[itemize]{wide=0pt, leftmargin=16pt, labelwidth=10pt, align=left}

\usepackage[subfigure]{tocloft}

\cftsetindents{section}{0em}{0em}

\makeatletter
\renewcommand{\cftsecpresnum}{\begin{lrbox}{\@tempboxa}}
\renewcommand{\cftsecaftersnum}{\end{lrbox}}
\makeatother

\graphicspath{{Images/}{../Images/}}

\usepackage[mathscr]{euscript}

\newcommand{\even}{\mathsf{even}}
\newcommand{\odd}{\mathsf{odd}}

\newcommand{\nc}{\mathsf{nc}}

\newcommand{\dd}{\mathrm{d}}

\newcommand{\bF}{\mathbb{F}}

\newcommand{\bN}{\mathbb{N}}

\newcommand{\bQ}{\mathbb{Q}}
\newcommand{\bR}{\mathbb{R}}

\newcommand{\bT}{\mathbb{T}}

\newcommand{\bZ}{\mathbb{Z}}

\newcommand{\cL}{\mathcal{L}}

\theoremstyle{definition}
\newtheorem{theorem}{Theorem}[section]
\newtheorem{proposition}[theorem]{Proposition}
\newtheorem{lemma}[theorem]{Lemma}
\newtheorem{corollary}[theorem]{Corollary}

\newtheorem*{theorem*}{Theorem}
\newtheorem*{proposition*}{Proposition}
\newtheorem*{lemma*}{Lemma}
\newtheorem*{corollary*}{Corollary}
\newtheorem*{definition*}{Definition}
\newtheorem*{example*}{Example}
\newtheorem*{remark*}{Remark}

\theoremstyle{remark}

\pgfkeys{/Dynkin diagram,
edge length=1.5cm,
fold radius=1cm,
indefinite edge/.style={
draw=black,
fill=white,
thin,
densely dashed}}

\usepackage{fancyhdr}
\newenvironment{red}{\relax\color{red}}{\hspace*{.5ex}\relax}
\newenvironment{blue}{\relax\color{blue}}{\hspace*{.5ex}\relax}
\newcommand{\ber}{\begin{red}}
\newcommand{\er}{\end{red}}
\newcommand{\beb}{\begin{blue}}
\newcommand{\eb}{\end{blue}}

%% file: src/1intro.tex
\section{Introduction}
\label{sec:intro}

Chebyshev's bias~\cite{chebyshev1853lettre} is a phenomenon in analytic number theory, referring to the observed preference for certain residue classes of primes in arithmetic progressions.
Specifically, it refers to the observed tendency for there to be more primes of the form \(4n + 3\) than of the form \(4n + 1\) among the primes up to some bound.
For general moduli, it is conjectured that non-quadratic residues are favored over quadratic residues.
Rubinstein and Sarnak were the first to prove this under the Generalized Riemann Hypothesis (GRH) and the Grand Simplicity Hypothesis (GSH) for Dirichlet \(L\)-functions.
Shanks~\cite{shanks1959quadratic} approached the bias from a different angle, examining higher-order effects.
Based on numerical experiments, he conjectured that there are more positive integers \(n\) such that \(\lambda(n) \chi_{-4}(n) = +1\) than those with \(\lambda(n) \chi_{-4}(n) = -1\), where \(\lambda\) is the Liouville function and \(\chi_{-4}\) is the nontrivial quadratic character modulo \(4\) (See Section~\ref{subsec:chebyshev} for details).

Cha~\cite{cha2008chebyshev} studied a function field analogue of Chebyshev's bias, closely following Rubinstein--Sarnak's argument~\cite{rubinstein1994chebyshev}, and proved the existence of bias under similar assumptions for \(L\)-functions.
He also exhibited examples with biases in unexpected directions when \(L\)-functions do not satisfy GSH.
Later works have extended these ideas to Galois extension of function fields~\cite{cha2011chebyshevglobal}, explicit formula for the limiting distributions of biases~\cite{cha2010biases}, multiple prime factors~\cite{devin2021chebyshev}, exceptional biases~\cite{bailleul2024exceptional}, and elliptic curves~\cite{cha2016prime}.

In this paper, we study the function field analogue of Shanks' conjecture.
Let \(p > 2\) be a prime and \(A = \bF_q[T]\) be the polynomial ring over a finite field \(\bF_q\) of size \(q = p^k\).
Let \(m \in A\) be a monic square-free polynomial and \(\chi_m\) be the unique nontrivial quadratic character modulo \(m\).
Writing the Liouville function on \(A\) as \(\lambda : A \to \{\pm 1\}\), we prove that \(\lambda \cdot \chi_m\) is biased towards \(+1\) under GSH.
This gives an answer to a function field analogue of Problem 22 in the Comparative Prime Number Theory Problem List~\cite{hamieh2024comparative}.

\begin{theorem*}[Corollary~\ref{cor:mainlambda} and~\ref{cor:gapLfunc}]
    Let \(A_{\pm}^\lambda(n;m)\) be the number of non-constant monic polynomials \(f \in A\) of degree at most \(n\) with \(\lambda(f) \chi_m(f) = \pm 1\).
    Assume that GSH holds for the Dirichlet \(L\)-function \(L(s, \chi_m)\).
    Then we have
    \[
    \lim_{n \to \infty} \frac{\#\{1 \le k \le n : A_{+}^\lambda(k;m) > A_{-}^\lambda(k;m)\}}{n} > \frac{1}{2}
    \]
    where a closed formula for the density is obtained.
    When \(q \equiv 1 \pmod{4}\), the density can be made arbitrarily close to \(\frac{1}{2}\) by choosing \(m\) appropriately with large enough degree.
\end{theorem*}
We also study several variants (e.g.\ non-cumulative densities and replacing \(\lambda\) by the M\"obius function \(\mu\)). In particular, we show there is no bias when restricting to square-free polynomials (Theorem~\ref{thm:mainmu}).
Example calculations of the densities are given in Section~\ref{sec:examples}, where the accompanying SageMath code can be found in the GitHub repository \href{https://github.com/seewoo5/sage-function-field}{github.com/seewoo5/sage-function-field}.

To prove our main theorem, we first exhibit \(L\)-functions attached to these biases (Proposition~\ref{prop:biasL}). We then apply the Kronecker--Weyl equidistribution theorem and Portmanteau's theorem to derive explicit density formulas via elementary trigonometric-series calculations. In particular, one expresses the bias as a linear combination of cosines with \(\bQ\)-linearly independent periods; equidistribution reduces the question to evaluating a constant term, and the final inequality follows from comparing symmetric polynomials in those cosines (Theorem~\ref{thm:blambdabias}).

\subsection*{Acknowledgments}

We thank Alexandre Bailleul, Byungchul Cha, Peter Humphries, and Greg Martin for helpful discussions.
We also thank the referee for careful reading and valuable comments.
This work is partially supported by NSF RTG grant DMS-2342225.

%% file: src/2preliminary.tex
\section{Preliminaries}
\label{sec:preliminaries}

\subsection{Chebyshev's bias and its variants}
\label{subsec:chebyshev}

Since Chebyshev first observed this bias in 1853~\cite{chebyshev1853lettre}, it has been widely studied by several authors.
Littlewood~\cite{littlewood1914distribution} proved that both \(\pi_+(n) > \pi_-(n)\) and \(\pi_+(n) < \pi_-(n)\) happens infinitely often, where \(\pi_{\pm}(n)\) is the number of primes \(p \le n\) such that \(p \equiv \pm 1 \pmod 4\).
Rubinstein and Sarnak~\cite{rubinstein1994chebyshev} gave more quantitative results.
Assuming the Generalized Riemann Hypothesis (GRH) and the Grand Simplicity Hypothesis (GSH), they proved that the \emph{logarithmic density} of the set of \(n\) such that \(\pi_-(n) > \pi_+(n)\) is about \(0.9959\), which is very close to \(1\) but not exactly.
In fact, they considered general ``prime race'' scenarios with general modulus where more than two congruence classes may exist, and computed the limiting distribution of the normalized bias.

Shanks~\cite{shanks1959quadratic} studied various aspects of Chebyshev's bias.
He considered the normalized bias
\[
\tau(n) = \frac{\pi_-(n) - \pi_+(n)}{\pi(n)} \sqrt{n}
\]
and conjectured that
\[
\lim_{N \to \infty} \frac{1}{N-1} \sum_{n=1}^{N} \tau(n) = 1
\]
supported by numerical evidence.
In particular, he considered the higher-order effects and conjectured that, if \(\pi_{\pm}^{(a)}(n)\) counts the number of positive integers of the form \(4k \pm 1\) that are product of \(a\) primes (not necessarily distinct), then the average of the order \(a\) normalized bias
\[
\tau^{(a)}(n) = \frac{\pi_-^{(a)}(n) - \pi_+^{(a)}(n)}{\pi_-^{(a)}(n) + \pi_+^{(a)}(n)} \sqrt{n}
\]
is \((-1)^{a}\)~\cite[Section 9]{shanks1959quadratic}. In other words, there are more numbers of the form \(4k-1\) (resp. \(4k + 1\)) with an odd (resp. even) number of prime factors than with an even (resp. odd) number of prime factors.
Specifically, he claimed that the number of \(n\) with \(\lambda(n) \chi_{-4}(n) = +1\) exceeds that of \(\lambda(n) \chi_{-4}(n) = -1\), where \(\lambda\) is the Liouville function
\[
\lambda(n) = (-1)^{e_1 + e_2 + \cdots + e_k} \,\,\text{for}\,\, n = p_1^{e_1} p_2^{e_2} \cdots p_k^{e_k}, \quad p_i's \text{ are distinct prime factors}
\]
and \(\chi_{-4} = \left(\frac{\cdot}{-4}\right)\) is the nontrivial quadratic Dirichlet character modulo \(4\). Note that the \(L\)-function for \(\lambda(n) \chi_{-4}(n)\) is simply given by 
\begin{equation}
    \label{eqn:LfuncShanks}
    \sum_{n \ge 1} \frac{\lambda(n) \chi_{-4}(n)}{n^{s}} = \frac{(1 - 2^{-2s}) \zeta(s)}{L(s, \chi_{-4})}
\end{equation}
and the bias might be related to the analytic properties of this \(L\)-function~\eqref{eqn:LfuncShanks}.
The Problem 22 in the comparative prime number theory problem list~\cite{hamieh2024comparative} asks to compute logarithmic density of the corresponding set (or replacing \(\chi_{-4}\) with other quadratic Dirichlet characters), assuming GRH and GSH for the relevant \(L\)-functions.

\subsection{Function fields, \texorpdfstring{\(L\)}{L}-functions, and Chebyshev's bias}

We fix an odd prime $p$.
Let $A = \bF_{q}[T]$ be the polynomial ring over a finite field $\bF_q$ of size $q = p^k$.
Let $A^\circ \subset A$ be the subset of monic polynomials.
For a monic square-free polynomial $m \in A^\circ$, let \( \chi_m(f) := \left(\frac{f}{m} \right) \) be the unique nontrivial quadratic character modulo \(m\).
M\"obius function \(\mu\) and Liouville function \(\lambda\) are defined on $A$ as
\begin{align}
    \mu(f) &:= \begin{cases}
        (-1)^{k} & f = g_1 \cdots g_k, \text{ $f$ is square-free and $g_1,\dots,g_k$ are distinct irreducible factors} \\
        0 & f \text{ is not square-free}
    \end{cases} \\
    \lambda(f) &:= (-1)^{e_1 + \cdots + e_k}, \,\text{if } f \text{ factors as } g_1^{e_1} \cdots g_k^{e_k}.
\end{align}
Especially, they agree on square-free polynomials.

The Riemann zeta function for \(A\) is defined as
\[
\zeta(s) = \sum_{f \in A^\circ} \frac{1}{|f|^{s}} = \frac{1}{1 - q^{1-s}}
\]
where \(|f| := q^{\deg f}\) for \(f \in A\).
The series converges for \(\Re(s) > 1\) and admits meromorphic continuation to the whole complex plane with a simple pole at \(s = 1\).
If \(\chi\) is a Dirichlet character on \(A\) with modulus \(m \in A\), then the Dirichlet \(L\)-function \(L(s, \chi)\) is similarly defined as
\[
L(s, \chi) = \sum_{f \in A^\circ} \frac{\chi(f)}{|f|^{s}}
\]
which also converges for \(\Re(s) > 1\) and admits meromorphic continuation to the whole complex plane.
We often write \(L\)-functions as a function in \(u = q^{-s}\) and write as \(\cL(u, \chi)\).
Both functions admit Euler product over monic irreducible polynomials as
\begin{align*}
    \zeta(s) &= \prod_{P} (1 - |P|^{-s})^{-1}, \\
    L(s, \chi) &= \prod_{P} (1 - \chi(P)|P|^{-s})^{-1}.
\end{align*}
When \(\chi\) is a nontrivial character, it is known that the Dirichlet \(L\)-function becomes a polynomial in \(u\) of degree at most \(\deg m - 1\)~\cite[Proposition 4.3]{rosen2013number}.
By the Riemann Hypothesis for function fields, the \emph{inverse of} zeros of \(\cL(u, \chi)\) are either \(u = 1\) (if this happens, it is always a simple zero) or have norm \(q^{\frac{1}{2}}\).
For quadratic characters, we have:
\begin{theorem}[{\cite[Proposition 6.4]{cha2008chebyshev}}]
    \label{thm:Lfuncquad}
    Let \(\chi\) be a quadratic Dirichlet character of modulus \(m\) and let \(M = \deg m\).
    \begin{enumerate}
        \item If \(M\) is odd, then \(\cL(u, \chi) \in \bZ[u]\) is a polynomial of degree \(M - 1\) where all the inverses of zeros have norm \(q^{\frac{1}{2}}\) and can be written as
        \[
            \cL(u, \chi) = \prod_{j=1}^{M'} (1 - \sqrt{q} e^{i\theta_j} u)(1 - \sqrt{q} e^{-i\theta_j} u) = \prod_{j=1}^{M'} (1 - 2 \sqrt{q} \cos \theta_j u + q u^2)
        \]
        for \(M' = (M-1)/2\) and \(\theta_1, \dots, \theta_{M'} \in \bR\).
        \item If \(M\) is even, then \(\cL(u, \chi) \in \bZ[u]\) is a polynomial of degree \(M - 1\) where it has a simple zero at \(u = 1\) and all the other inverse zeros have norm \(q^{\frac{1}{2}}\).
        It can be written as
        \[
        \cL(u, \chi) = (1 - u) \prod_{j=1}^{M'} (1 - \sqrt{q} e^{i\theta_j} u)(1 - \sqrt{q} e^{-i\theta_j} u) = (1 - u) \prod_{j=1}^{M'} (1 - 2 \sqrt{q} \cos \theta_j u + q u^2)
        \]
        for \(M' = (M-2) / 2\) and \(\theta_1, \dots, \theta_{M'} \in \bR\).
    \end{enumerate}
\end{theorem}

In~\cite{cha2008chebyshev}, Cha proved a version of Chebyshev's bias in function fields of odd characteristic, by following the proof of Rubinstein--Sarnak~\cite{rubinstein1994chebyshev} in the number field case.
As in \emph{loc. cit.}, he defined a function field version of the Grand Simplicity Hypothesis (GSH), which is on \(\bQ\)-linear independence of arguments of inverse zeros of \(L\)-functions.
Under GSH, he showed that quadratic non-residues are preferred over quadratic residues modulo irreducible polynomials.
One noticeable difference from the number field case is that, in the function field case, there are examples where GSH does not hold.
In such case, the bias could happen in both expected and unexpected directions~\cite[Section 5]{cha2008chebyshev}.

As mentioned previously, there are several variants of Chebyshev's bias over function fields \cite{cha2010biases,cha2011chebyshevglobal,cha2016prime,devin2021chebyshev,bailleul2024exceptional}.
Most notably, Devin and Meng~\cite{devin2021chebyshev} studied the bias for the polynomials with a fixed number of irreducible factors.
For each integer \(k \ge 1\) and a square-free monic polynomial \(m\), they proved that the monic polynomials with \(k\) irreducible factors tend to be quadratic non-residues (resp. quadratic residues) modulo \(m\) when \(k\) is odd (resp. \(k\) is even), assuming GSH for all Dirichlet \(L\)-functions of quadratic characters modulo \(m\).
In particular, it recovers Cha's result~\cite{cha2008chebyshev} when \(k = 1\).
Shanks' bias \cite{shanks1959quadratic} may be viewed as a case where we consider all \(k \ge 1\) at once.
However, the asymptotic formula for the normalized bias \cite[Theorem 1.2]{devin2021chebyshev} is only shown for \(k = o((\log N)^{1/2})\), where \(N\) is the maximal degree of polynomials being considered, which may not be enough to deduce Shanks' bias since the typical number of irreducible factors of polynomials of degree \(N\) is about \(\log N\) (See also Remark 2 or \emph{loc. cit.}).
As we will see in Section \ref{sec:Lfunc} (Proposition \ref{prop:biasL}), the \(L\)-function associated to the Shanks' bias is simple and this allows us to compute the associated densities of degrees (Corollary \ref{cor:mainlambda}).
Also, the case of the bias with M\"obius function (i.e. restrict to the square-free polynomials) has not been exploited in \emph{loc. cit}.


\subsection{Kronecker--Weyl Equidistribution theorem}

The Kronecker--Weyl theorem is a following multidimensional generalization of the Weyl's equidistribution theorem.

\begin{theorem}[Kronecker--Weyl]
\label{thm:Kronecker-Weyl}
    Let \(g : \bT^d \to \bR\) be a continuous function on the \(d\)-dimensional torus \(\bT^d = (\bR / 2 \pi \bZ)^d\).
    Let \(\theta_1, \dots, \theta_d \in \bR\) be the real numbers such that \(\pi, \theta_1, \dots, \theta_d\) are \(\bQ\)-linearly independent real numbers.
    Then \(\{(n \theta_1, \dots, n \theta_d) : n \in \bN\}\in \bT^d\) is equidistributed and
    \begin{equation}
        \label{eqn:KWlim}
        \lim_{N \to \infty} \frac{1}{N} \sum_{n=1}^{N} g(n \theta_1, \dots, n \theta_d) = \int_{\bT^d} g(\phi) \dd \phi
    \end{equation}
    where \(\dd \phi\) is the normalized Haar measure on \(\bT^d\) with \(\int_{\bT^d} \dd \phi = 1\).
\end{theorem}
A proof can be found in the Appendix of~\cite{humphries2012mertens}.
In such case, the counting measure
\[
\dd\phi_n(S) := \frac{1}{n}\# \{1 \le k \le n : (k \theta_1, \dots, k \theta_d) \in S\}
\]
weakly converges to \(\dd\phi\), so by Portmanteau's theorem~\cite[Theorem 2.1]{billingsley2013convergence}, we have
\[
\lim_{n \to \infty} \dd\phi_n(S) = \dd\phi(S)
\]
for any Borel set \(S \subset \bT^d\) with \(\dd\phi(\partial S) = 0\).
From this, we can deduce the following corollary, which will be used later to compute various densities.
\begin{corollary}
    \label{cor:sindensity}    
    Let \(\theta_1, \dots, \theta_d \in \bR\) be real numbers where \(\{\pi, \theta_1, \theta_2, \dots, \theta_d\}\) are \(\bQ\)-linearly independent.
    Let \(\alpha, \beta_1, \dots, \beta_d, \omega_1, \dots, \omega_d \in \bR\).
    For any sequence \(x_n \to 0\), we have 
    \begin{equation}
        \label{eqn:sindensity}
        \lim_{n \to \infty} \frac{1}{n} \#\left\{1 \le k \le n : \alpha + \sum_{j=1}^{d} \beta_j \sin(n\theta_j + \omega_j) > x_n\right\} = \dd \phi\left(\left\{\alpha + \sum_{j=1}^{d} \beta_j \sin(y_j) > 0\right\}\right)
    \end{equation}
\end{corollary}
\begin{proof}
    When \(x_n\) is identically \(0\), this follows from the previous discussion applied to the Borel set
    \[
        S = g^{-1}((0, \infty)),\quad g(y_1, \dots, y_d) = \alpha + \sum_{j=1}^{d} \beta_j \sin(y_j)
    \]
    where \(g\) is continuous on \(\bT^d\).
    The boundary \(\partial S\) is contained in \(g^{-1}(\{0\})\), which has measure zero since it is the zero set of a nontrivial analytic function.
    For general \(x_n\), fix \(\epsilon > 0\) and let \(N \in \bN\) be such that \(|x_n| < \epsilon\) for all \(n \ge N\).
    Then \(g(y_1, \dots, y_d) > \epsilon\) implies \(g(y_1, \dots, y_d) > x_n\) and the limit~\eqref{eqn:sindensity} is bounded above by
    \[
    \lim_{n \to \infty} \frac{1}{n} \#\left\{1 \le k \le n : \alpha + \sum_{j=1}^{d} \beta_j \sin(n\theta_j + \omega_j) > \epsilon\right\} = \dd \phi\left(\left\{\alpha - \epsilon + \sum_{j=1}^{d} \beta_j \sin(y_j) > 0\right\}\right).
    \]
    Similarly, we obtain a lower bound
    \[
    \lim_{n \to \infty} \frac{1}{n} \#\left\{1 \le k \le n : \alpha + \sum_{j=1}^{d} \beta_j \sin(n\theta_j + \omega_j) > -\epsilon\right\} = \dd \phi\left(\left\{\alpha + \epsilon + \sum_{j=1}^{d} \beta_j \sin(y_j) > 0\right\}\right).
    \]
    Since \(\epsilon\) is arbitrary, we can take \(\epsilon \to 0\) and obtain the desired result.
\end{proof}

%% file: src/3lfunc.tex
\section{\texorpdfstring{\(L\)}{L}-functions}
\label{sec:Lfunc}

To study Shanks' bias, we consider the counts
\begin{align}
    a_{\pm}^{\mu}(n; m) &:= \# \{f \in A^\circ : \deg f = n,\,\mu(f) \chi_m(f) = \pm 1\}, \\
    a_{\pm}^{\lambda}(n; m) &:= \#\{ f \in A^\circ : \deg f = n,\,\lambda(f) \chi_m(f) = \pm 1 \},
\end{align}
for each \(n \ge 1\).
Define the \emph{cumulative} counts
\begin{equation}
    A_{\pm}^{\bullet}(n; m) := \sum_{k = 1}^{n} a_{\pm}^{\bullet}(k; m)
\end{equation}
for \( \bullet \in \{\mu, \lambda\}\).
Then our goal is to analyze the biases
\begin{align}
    b^\bullet(n;m) &:= a_+^\bullet(n;m) - a_-^\bullet(n;m) \\
    B^\bullet(n;m) &:= A_+^\bullet(n;m) - A_-^\bullet(n;m),
\end{align}
i.e. how often do we have \(b^\bullet(n; m) > 0\) or \(B^\bullet(n;m) > 0\).

The following proposition shows that the biases have nice \(L\)-functions.
\begin{proposition}
    \label{prop:biasL}
    Let
    \begin{align}
        L^\lambda(s; m) &:= \sum_{f \in A^\circ} \frac{\lambda(f) \chi_m(f)}{|f|^{s}}, \\
        L^\mu(s; m) &:= \sum_{f \in A^\circ} \frac{\mu(f) \chi_m(f)}{|f|^{s}}.
    \end{align}
    and define \(\cL^\bullet(u; m) := L^\bullet(s; m)\) for \(u = q^{-s}\) and \(\bullet \in \{\lambda, \mu\}\).
    Then 
    \begin{align}
        \cL^\lambda(u; m) &= \sum_{n \ge 0} b^\lambda(n;m) u^n \label{eqn:cLlambda} \\
        \cL^\mu(u; m) &= \sum_{n \ge 0} b^\mu(n;m) u^n  \label{eqn:cLmu}
    \end{align}
    and they admit Euler factorizations
    \begin{align}
        L^\lambda(s; m) &= \prod_{P} \left(1 + \frac{\chi_m(P)}{|P|^{s}}\right)^{-1} \label{eqn:Llambda} \\
        L^\mu(s; m) &= \prod_{P} \left(1 - \frac{\chi_m(P)}{|P|^{s}} \right) = \frac{1}{L(s, \chi_m)}. \label{eqn:Lmu}
    \end{align}
\end{proposition}
\begin{proof}
    We have
    \[
        \cL^\lambda(u;m) = \sum_{n \ge 0} \left(\sum_{\substack{f \in A^\circ \\ \deg f = n}}\lambda(f) \chi_m(f)\right) u^{n}
    \]
    and the \(n\)-th coefficient of \(\cL^\lambda(u;m)\) is \(a_+^\lambda(n;m) - a_-^\lambda(n;m) = b^\lambda(n;m)\) by definition.
    Moreover, since \(\lambda\) and \(\chi_m\) are multiplicative, \(L^\lambda(s;m)\) admits an Euler factorization
    \[
        L^\lambda(s;m) = \prod_P \left(1 - \frac{\chi_m(P)}{|P|^s} + \frac{\chi_m(P^2)}{|P|^{2s}} - \cdots\right) = \prod_P \left(1 + \frac{\chi_m(P)}{|P|^{s}}\right)^{-1}.
    \]
    The identity~\eqref{eqn:cLmu} for \(\cL^\mu(u;m)\) follows similarly, while the Euler factorization~\eqref{eqn:Lmu} comes from the multiplicativity of $\mu, \chi_m$ and $\mu(P^2) = 0$.
\end{proof}

As a corollary, we obtain a simple expression for \(L^\lambda(s;m)\) analogous to~\eqref{eqn:LfuncShanks}.
\begin{corollary}
    \label{cor:Llambda}
    Write \(m = m_1 \cdots m_r\) as a product of distinct irreducible monic polynomials \(m_i\) in \(A\) and let \(M_i := \deg m_i\).
    Then
    \begin{equation}
        \cL^\lambda(u; m) = \frac{\prod_{i=1}^{r} (1 - u^{2M_i})}{1 - qu^{2}}\, \frac{1}{\cL(u, \chi_m)}.
    \end{equation}
    In particular, this shows that \(\cL^\lambda(u; m)\) is a rational function in \(u\).
\end{corollary}
\begin{proof}
    \begin{align*}
        \frac{L^\lambda(s; m)}{L^\mu(s; m)} &= \prod_{P} \left(1 - \frac{\chi_m(P)^2}{|P|^{2s}}\right)^{-1} \\
        &= \prod_{P \nmid m} \left(1 - \frac{1}{|P|^{2s}}\right)^{-1} \\
        &= \zeta(2s) \cdot \prod_{i=1}^{r} (1 - |m_i|^{-2s}) \\
        &= \frac{\prod_{i=1}^{r} (1 - u^{2M_i})}{1 - qu^2}
    \end{align*}
    and the results follows from~\eqref{eqn:Lmu}.
\end{proof}

\begin{corollary}
    \label{cor:Llambdacum}
    The generating functions for the cumulative biases are given by
    \begin{align}
        \sum_{n\ge1}B^\bullet(n;m) u^n = \frac{\cL^\bullet(u;m)-1}{1-u},\quad \bullet\in\{\mu,\lambda\}.
    \end{align}
\end{corollary}
\begin{proof}
Both follow from the following identity
\[
\frac{\sum_{n \ge 1} a_n u^n}{1 - u} = \left(\sum_{n \ge 1} a_n u^n\right) (1 + u + u^2 + \cdots) = \sum_{n \ge 1} \left(\sum_{1 \le k \le n} a_k \right) u^n.
\]
\end{proof}

%% file: src/4density.tex
\section{Density and biases}
\label{sec:density}

The natural densities for the degree of polynomials are defined as
\begin{align*}
\delta_+^\bullet(m) := \lim_{n \to \infty}\frac{\#\{1 \le k \le n : A_+^\bullet(k; m) - A_-^\bullet(k; m) > 0\}}{n}, \\
\delta_0^\bullet(m) := \lim_{n \to \infty}\frac{\#\{1 \le k \le n : A_+^\bullet(k; m) - A_-^\bullet(k; m) = 0\}}{n}, \\
\delta_-^\bullet(m) := \lim_{n \to \infty}\frac{\#\{1 \le k \le n : A_+^\bullet(k; m) - A_-^\bullet(k; m) < 0\}}{n}.
\end{align*}
for \(\bullet \in \{\mu, \lambda\}\).
We also consider the \emph{non-cumulative} version of the densities as
\begin{align*}
\delta_+^{\bullet, \nc}(m) := \lim_{n \to \infty}\frac{\#\{1 \le k \le n : a_+^\bullet(k; m) - a_-^\bullet(k; m) > 0\}}{n}, \\
\delta_0^{\bullet, \nc}(m) := \lim_{n \to \infty}\frac{\#\{1 \le k \le n : a_+^\bullet(k; m) - a_-^\bullet(k; m) = 0\}}{n}, \\
\delta_-^{\bullet, \nc}(m) := \lim_{n \to \infty}\frac{\#\{1 \le k \le n : a_+^\bullet(k; m) - a_-^\bullet(k; m) < 0\}}{n}.
\end{align*}

The following lemma is helpful for the computations that follow.
\begin{lemma}
    \label{lem:sinasum}
    For \(a, \theta, \omega \in \bR\), we have
    \begin{equation}
    \label{eqn:sinasum}
        \sum_{k=1}^{n} a^k \sin(k \theta + \omega) = \frac{a \sin(\theta + \omega) - a^2 \sin \omega - a^{n+1} \sin((n+1)\theta + \omega) + a^{n+2} \sin(n\theta + \omega)}{1 - 2a \cos \theta + a^2}.
    \end{equation}
    In particular, we have
    \begin{align}
        \sin(2\theta) + \sin(4\theta) + \cdots + \sin(2k\theta) &= \frac{\cos\theta - \cos((2k+1)\theta)}{2\sin\theta} \label{eqn:sinsumeven}\\
        \sin \theta + \sin (3 \theta) + \cdots + \sin((2k+1)\theta) &= \frac{1 - \cos((2k+2)\theta)}{2 \sin \theta} \label{eqn:sinsumodd}
    \end{align}
\end{lemma}
\begin{proof}
    Use \(a^k \sin(k\theta + \omega) = a^k \Im (e^{i(k\theta + \omega)}) = \Im((a e^{i\theta})^k \cdot e^{i\omega})\) and the geometric series formula.
    The last two equations easily follow from~\eqref{eqn:sinasum}.
\end{proof}

Let us consider the \(\mu\)-bias first.
Since the \(L\)-functions for \(b^\lambda(n;m)\) and \(B^\lambda(n;m)\) are rational functions, they satisfy (\(M\)-term) linear recurrence relations.
Thus, one can express \(b^\lambda(n;m)\) and \(B^\lambda(n;m)\) as linear combinations of trigonometric functions.
\begin{proposition}
\label{prop:bmu}
    Let \(m\in A\) be a monic square-free polynomial of degree \(M\) and \(\chi_m\) be the unique nontrivial quadratic character modulo \(m\).
    Define \(M'\) and \(\theta_1, \dots, \theta_{M'}\) as in Theorem~\ref{thm:Lfuncquad}.
    Then we have
    \begin{equation}
        \label{eqn:bmu}
        b^\mu(n;m) = q^{\frac{n}{2}}\sum_{j=1}^{M'} \gamma_j \sin(n \theta_j + \phi_j) + \begin{cases}
            \delta & \text{if }2 \mid M \\
            0 & \text{if }2 \nmid M
        \end{cases}
    \end{equation}
    and
    \begin{equation}
    \begin{aligned}
        \label{eqn:Bmu}
        B^\mu(n;m) &= \sum_{j=1}^{\frac{M-1}{2}} \gamma_j  \frac{q^{\frac{1}{2}} \sin(\theta_j + \phi_j) - q \sin \phi_j - q^{\frac{n+1}{2}} \sin((n+1)\theta_j + \phi_j) + q^{\frac{n+2}{2}} \sin(n\theta_j + \phi_j)}{1 - 2q^{\frac{1}{2}} \cos \theta_j + q} \\
        &\quad+ \begin{cases}
            n \delta & \text{if }2 \mid M \\
            0 & \text{if }2 \nmid M
        \end{cases}
    \end{aligned}
    \end{equation}
    for some \(\gamma_1, \dots, \gamma_{M'}, \phi_1, \dots, \phi_{M'}, \delta \in \bR\).
\end{proposition}
\begin{proof}
    If \(M\) is odd, then by Proposition~\ref{prop:biasL}, \(b^\mu(n;m)\) satisfies the homogenous linear recurrence relation
    \[
    \sum_{k=0}^{M-1} a_k b^\mu(n-k;m) = 0 
    \]
    for $n \ge 1$, when \(\cL(u, \chi_m) = \sum_{k=0}^{M-1} a_k u^k\). The characteristic polynomial of this recurrence relation is the reciprocal polynomial of \(\cL(u, \chi_m)\), which has roots \(q^{\frac{1}{2}} e^{i\theta_j}\) and \(q^{\frac{1}{2}} e^{-i\theta_j}\) for \(j = 1, \dots, M'\).
    Hence \(b^\mu(n;m)\) can be expressed as a linear combination of \(q^{\frac{n}{2}} e^{in\theta_j}\) and \(q^{\frac{n}{2}} e^{-in\theta_j}\), and since they are all real, we can write it as a linear combination of \(q^{\frac{n}{2}}\sin(n\theta_j + \phi_j)\) for some \(\phi_j \in \bR\).
    For even \(M\), we have an extra term \(\delta \in \bR\) coming from the zero at \(u=1\) of \(\cL(u, \chi_m)\).
    ~\eqref{eqn:Bmu} follows from~\eqref{eqn:sinasum}.
\end{proof}

Using this, we can show that there's no Shanks' bias if we consider only square-free polynomials under GSH.
\begin{theorem}
    \label{thm:mainmu}
    Assume that \(\cL(u, \chi_m)\) satisfies GSH.
    Then
    \begin{align}
        \delta_+^{\mu, \nc}(m) = \delta_-^{\mu, \nc}(m) = \frac{1}{2},&\, \delta_0^{\mu, \nc}(m) = 0, \label{eqn:deltamunc} \\
        \delta_+^{\mu}(m) = \delta_-^{\mu}(m) = \frac{1}{2},&\, \delta_0^{\mu}(m) = 0. \label{eqn:deltamu}
    \end{align}
\end{theorem}
\begin{proof}
    The non-cumulative densities~\eqref{eqn:deltamunc} follow from Corollary~\ref{cor:sindensity} and~\eqref{eqn:bmu}, by taking $x_n = - \delta q^{-\frac{n}{2}}$.
    For the cumulative densities, assume \(M\) is odd. By~\eqref{eqn:Bmu}, \(B^\mu(n;m) > 0\) if and only if
    \begin{align}
        &C + q^{\frac{n+1}{2}} \sum_{j=1}^{\frac{M-1}{2}} \frac{-\sin\theta_j\cos(n\theta_j + \omega_j) + (q^{\frac{1}{2}} - \cos\theta_j)\sin(n\theta_j + \omega_j)}{1 - 2q^{\frac{1}{2}} \cos \theta_j + q} \nonumber \\
        &=C + q^{\frac{n+1}{2}} \sum_{j=1}^{\frac{M-1}{2}} \frac{\sin(n \theta_j + \omega_j')}{\sqrt{1 - 2q^{\frac{1}{2}}\cos \theta_j + q}} > 0 \nonumber \\
        &\Leftrightarrow \sum_{j=1}^{\frac{M-1}{2}} \frac{\sin(n \theta_j + \omega_j')}{\sqrt{1 - 2q^{\frac{1}{2}}\cos \theta_j + q}} > - C q^{-\frac{n+1}{2}}
    \end{align}
    where
    \[
        C = \sum_{j=1}^{\frac{M-1}{2}} \gamma_j \frac{q^{-\frac{1}{2}} \sin(\theta_j + \omega_j) - \sin \omega_j}{1 - 2 q^{\frac{1}{2}} \cos\theta_j + q}.
    \]
    Since \(-Cq^{-\frac{n+1}{2}} \to 0\) as \(n \to \infty\), we can apply Corollary~\ref{cor:sindensity} and obtain~\eqref{eqn:deltamu}.
    A similar argument works for even \(M\) as well, where the extra \(n\delta\) term does not affect the limit since \(n q^{-\frac{n+1}{2}} \to 0\) as \(n \to \infty\).
\end{proof}

For \(\lambda\)-bias, we can also express \(b^\lambda(n;m)\) and \(B^\lambda(n;m)\) as linear combinations of trigonometric functions. We have an extra non-constant term that becomes a source of bias.
\begin{theorem}
\label{thm:blambdabias}
    Assume that \(\cL(u, \chi_m)\) satisfies GSH. Let \(M \ge 3\).
    \begin{enumerate}
        \item Assume \(M\) is odd and let \(M' = \frac{M-1}{2}\). Then
        \begin{align}
            \label{eqn:blambda}
            b^\lambda(n;m) = q^{\frac{n}{2}}\left(\alpha_n + \sum_{j=1}^{M'} \beta_j \sin(n \theta_j + \omega_j)\right)
        \end{align}
        for all \(n \ge 2 M\), where \(\beta_j, \omega_j \in \bR\) are constants independent of \(n\).
        Also, \(\alpha_n\) can be expressed as
        \begin{align}
            \label{eqn:alphan}
            \alpha_n = \frac{\prod_{i=1}^{t} (1 - q^{-M_i})}{2^{M'} \prod_{j=1}^{M'} \sin^2 \theta_j} \times\begin{cases}
                e_{\even}^{(M')} & \text{if }2 \mid n \\
                e_{\odd}^{(M')} & \text{if }2 \nmid n
            \end{cases}
        \end{align}
        where \(e_k^{(M')}\) is a \(k\)-th elementary symmetric polynomial in \(M'\) variables \(\cos \theta_1, \dots, \cos\theta_{M'}\) and
        \begin{align}
            e_{\even}^{(M')} &= e_0^{(M')} + e_2^{(M')} + e_4^{(M')} + \cdots \label{eqn:eeven} \\
            e_{\odd}^{(M')} &= e_1^{(M')} + e_3^{(M')} + e_5^{(M')} + \cdots \label{eqn:eodd}
        \end{align}
        \item Assume \(M\) is even and let \(M' = \frac{M-2}{2}\). Then
        \begin{align}
            \label{eqn:blambdaevenM}
            b^\lambda(n;m) = q^{\frac{n}{2}} \left[\alpha_n' + \sum_{j=1}^{M'} \beta_j \sin(n\theta_j + \omega_j)\right] + C
        \end{align}
        holds for all \(n \ge 2M\), where \(C, \beta_j, \omega_j \in \bR\) are constants independent of \(n\) and \(\alpha_n'\) can be written as
        \begin{align}
            \label{eqn:alphanevenM}
            \alpha_n' = \frac{\prod_{i=1}^{t} (1 - q^{-M_i})}{2^{M'} \prod_{j=1}^{M'} \sin^2 \theta_j} \times \begin{cases}
                e_\even^{(M')} \frac{q}{q-1} + e_\odd^{(M')} \frac{q^{\frac{1}{2}}}{q-1} & \text{if }2 \mid n \\
                e_\even^{(M')} \frac{q^{\frac{1}{2}}}{q-1} + e_\odd^{(M')} \frac{q}{q-1} & \text{if }2 \nmid n
            \end{cases}
        \end{align}
        with $e_\even$ and $e_\odd$ defined as in~\eqref{eqn:eeven} and~\eqref{eqn:eodd}, respectively.
    \end{enumerate}
\end{theorem}
\begin{proof}
    Assume \(M\) is odd. By Theorem \ref{thm:Lfuncquad} and Corollary \ref{cor:Llambda}, we have
    \[
        \cL^\lambda(q^{-\frac{1}{2}}u;m) = \frac{\prod_{i=1}^{r} (1 - q^{-M_i}u^{2M_i})}{1 - u^2} \prod_{j=1}^{M'} \left( \sum_{n \ge 0} \frac{\sin((n+1)\theta_j)}{\sin\theta_j} u^n\right).
    \]
    Consider the partial product
    \begin{equation}
        \cL^\lambda_{t, N}(u;m) = \frac{\prod_{i=1}^{t}(1 - q^{-M_i}u^{2M_i})}{1 - u^2} \prod_{j=1}^{N} \left( \sum_{n \ge 0} \frac{\sin((n+1)\theta_j)}{\sin\theta_j} u^n\right) = \sum_{n \ge 0} b_{t, N}^\lambda(n;m) u^n.
    \end{equation}
    for \(1 \le t \le r\) and \(N \ge 1\).
    Using induction on \(t\) and \(N\), we will prove that \(b_{t, N}^\lambda(n;m)\) has a form of
    \begin{equation}
    \label{eqn:bN}
        b_{t, N}^\lambda(n;m) = \alpha_{t,N,n} + \sum_{j=1}^{N} \beta_{j,t,N} \sin(n \theta_{j} + \omega_{j,t,N})
    \end{equation}
    for \(n \ge 2M_1 + \cdots + 2M_t\), where \(\beta_{j, t, N}, \omega_{j, t, N} \in \bR\) do not depend on \(n\) and
    \begin{equation}
    \label{eqn:alphaN}
        \alpha_{t,N,n} = \frac{\prod_{i=1}^{t} (1 - q^{-M_i})}{2^{N} \prod_{j=1}^{N} \sin^2 \theta_j} \times\begin{cases}
            e_0^{(N)} + e_2^{(N)} + e_4^{(N)} + \cdots & \text{if }2 \mid n \\
            e_1^{(N)} + e_3^{(N)} + e_5^{(N)} + \cdots & \text{if }2 \nmid n
        \end{cases}
    \end{equation}
    where \(e_k^{(N)}\) is a \(k\)-th elementary symmetric polynomial in \(N\) variables \(\cos \theta_1, \dots, \cos \theta_N\).

    When \(t = N = 1\),
    \begin{align*}
        \cL^\lambda_{1,1}(u;m) &= \frac{1 - q^{-M_1}u^{2M_1}}{1 - u^2} \sum_{n \ge 0} \frac{\sin((n+1)\theta_1)}{\sin \theta_1} u^n \\
        &= (1 - q^{-M_1}u^{2M_1}) \left[ \sum_{k \ge 0} \frac{(\sin \theta_1 + \sin (3\theta_1) + \cdots + \sin ((2k+1)\theta_1)}{\sin \theta_1} u^{2k} \right. \\
        &\quad+\left. \frac{(\sin (2 \theta_1) + \sin(4\theta_1) + \cdots + \sin((2k+2)\theta_1)}{\sin\theta_1} u^{2k + 1}\right] \\
        &= (1 - q^{-M_1}u^{2M_1}) \left[ \sum_{k \ge 0} \left(\frac{1 - \cos((2k+2)\theta_1)}{2\sin^2 \theta_1}\right) u^{2k} + \left(\frac{\cos \theta_1 - \cos((2k+3)\theta_1)}{2\sin^2 \theta_1}\right) u^{2k+1}  \right]
    \end{align*}
    so for \(n \ge 2M_1\), we have
    \[
        b_{1,1}^\lambda(n;m) = \begin{cases}
            \dfrac{1 - q^{-M_1} - \cos((n+2)\theta_1) + q^{-M_1} \cos((n-2M_1+2)\theta_1)}{2\sin^2\theta_1} & \text{if }2 \mid n \\
            \dfrac{(1 - q^{-M_1}) \cos \theta_1 - \cos((n+2)\theta_1) + q^{-M_1} \cos((n-2M_1+2)\theta_1)}{2\sin^2\theta_1} & \text{if }2 \nmid n
        \end{cases}
    \]
    and this can be rewritten as a form of
    \[
        b_{1,1}^\lambda(n;m) = \begin{cases}
            \frac{1 - q^{-M_1}}{2 \sin^2\theta_1} + \beta_{1, 1, 1} \sin (n \theta_{1} + \omega_{1, 1, 1}) & \text{if }2 \mid n \\
            \frac{1 - q^{-M_1}}{2 \sin^2\theta_1} \cdot \cos \theta_1 + \beta_{1, 1, 1} \sin (n \theta_{1} + \omega_{1, 1, 1}) & \text{if }2 \nmid n
        \end{cases}
    \]
    where
    \[
        \beta_{1, 1, 1} = \frac{\sqrt{1 - 2 q^{-M_1} \cos(M_1\theta_1) + q^{-2M_1}}}{2 \sin^2 \theta_1}.
    \]
    Now, assume that \eqref{eqn:bN} and \eqref{eqn:alphaN} holds for some \(t < r\) and \(N < M'\).
    First, we show that it holds for \(t\) and \(N+1\).
    By the induction hypothesis, we have
    \begin{align*}
        \cL_{t,N+1}^\lambda(u;m) &= \cL_{t,N}^\lambda(u;m) \sum_{n \ge 0} \frac{\sin((n+1)\theta_{N+1})}{\sin \theta_{N+1}} u^n \\
        &= \left(\sum_{n \ge 0} b_{t,N}^\lambda(n;m) u^n\right)\left(\sum_{n \ge 0} \frac{\sin((n+1)\theta_{N+1})}{\sin \theta_{N+1}} u^n\right)
    \end{align*}
    and
    \begin{align*}
        b_{t,N+1}^\lambda(n;m) &= \sum_{k=0}^{n} b_{t,N}^\lambda(k;m) \cdot \frac{\sin((n+1-k)\theta_{N+1})}{\sin\theta_{N+1}} \\
        &=\sum_{k=0}^{n} \left(\alpha_{t,N,k} + \sum_{j=1}^{N} \beta_{j,t,N} \sin(k\theta_{j} + \omega_{j,t,N})\right) \frac{\sin((n+1-k)\theta_{N+1})}{\sin \theta_{N+1}}.
    \end{align*}
    We have
    \begin{align*}
        &\sin(k\theta_{j} + \omega_{j,t,N}) \sin((n+1-k) \theta_{N+1}) \\
        &= \frac{1}{2} (\cos(k(\theta_{j} + \theta_{N+1})  - (n+1)\theta_{N+1}+ \omega_{j, t, N}) - \cos(k(\theta_{j} - \theta_{N+1}) + (n+1) \theta_{N+1} + \omega_{j, t, N}))
    \end{align*}
    and the summation over the product of sines becomes the sum of
    \begin{align*}
        &\frac{\beta_{j,t,N}}{2 \sin \theta_{N+1}} \sum_{k=0}^{n} \cos(k(\theta_{j} + \theta_{N+1}) - (n+1)\theta_{N+1} + \omega_{j, t, N}) - \cos(k(\theta_{j} - \theta_{N+1}) + (n+1) \theta_{N+1} + \omega_{j, t, N}) \\
        &= \frac{\beta_{j,t,N}}{4\sin\theta_{N+1}} \left[\frac{\sin\left(n \theta_j + \frac{\theta_j - \theta_{N+1}}{2} + \omega_{j,t,N}\right) + \sin\left(n\theta_{N+1} + \frac{\theta_j + 3 \theta_{N+1}}{2} - \omega_{j,t,N}\right)}{\sin\left(\frac{\theta_{j} + \theta_{N+1}}{2}\right)}\right. \\
        &\quad \left. - \frac{\sin\left(n\theta_j + \frac{\theta_j + \theta_{N+1}}{2} + \omega_{j,t,N}\right) - \sin\left(n\theta_{N+1} - \frac{\theta_j - 3 \theta_{N+1}}{2} + \omega_{j,t,N}\right)}{\sin\left(\frac{\theta_{j} - \theta_{N+1}}{2}\right)}\right]
    \end{align*}
    for \(j = 1, \dots, M'\).
    These (four) terms can be subsumed into the oscilating terms \(\beta_{j,t,N+1} \sin(n\theta_{j} + \omega_{j,t,N+1})\) and \(\beta_{N+1,t,N+1} \sin (n\theta_{N+1} + \omega_{N+1,t,N+1})\) of \eqref{eqn:bN} for \(N+1\), by applying the identity
    \begin{equation}
        \label{eqn:sinadd}
        \begin{aligned}
        a \sin(\theta + \alpha) + b \sin(\theta + \beta) &= c \sin(\theta + \gamma),\quad \text{where} \\
        c = \sqrt{a^2 + b^2 + 2ab \cos(\beta - \alpha)},\quad \gamma &= \tan^{-1}\left(\frac{a \sin \alpha + b \sin \beta}{a \cos \alpha + b \cos \beta}\right)
        \end{aligned}
    \end{equation}
    repeatedly.
    For the remaining sum over \(\alpha_{t,N,k} \frac{\sin ((n+1-k) \theta_{N+1})}{\sin\theta_{N+1}}\), we consider the cases when \(n\) is even or odd separately.
    Let \(e_{\even}^{(N)} = \sum_{0 \le k \le \frac{N}{2}} e_{2k}^{(N)} \) and \(e_{\odd}^{(N)} = \sum_{0 \le k \le \frac{N-1}{2}} e_{2k+1}^{(N)}\).
    If \(2 \mid n\) and \(n = 2n'\), 
    \begin{align*}
        &\sum_{k=0}^{n} \alpha_{N,k} \frac{\sin ((n+1-k) \theta_{N+1})}{\sin\theta_{N+1}} \\
        &= \sum_{k=0}^{n'} \alpha_{N,2k} \frac{\sin((n+1-2k) \theta_{N+1})}{\sin\theta_{N+1}} + \sum_{k=0}^{n'-1} \alpha_{N,2k+1} \frac{\sin((n-2k') \theta_{N+1})}{\sin\theta_{N+1}} \\
        &= \frac{\prod_{i=1}^{t}(1 - q^{-M_i}) e_{\even}^{(N)}}{2^N \prod_{j=1}^{N} \sin^2 \theta_j} \sum_{k=0}^{n'} \frac{\sin((2n'+1-2k)\theta_{N+1})}{\sin\theta_{N+1}} + \frac{\prod_{i=1}^{t}(1 - q^{-M_i}) e_{\odd}^{(N)}}{2^N \prod_{j=1}^{N} \sin^2 \theta_j} \sum_{k=0}^{n'} \frac{\sin((2n'-2k)\theta_{N+1})}{\sin\theta_{N+1}} \\
        &= \frac{\prod_{i=1}^{t}(1 - q^{-M_i}) e_{\even}^{(N)}}{2^{N+1} \prod_{j=1}^{N+1} \sin^2 \theta_j}(1 - \cos((2n'+2)\theta_{N+1})) + \frac{\prod_{i=1}^{t}(1 - q^{-M_i}) e_{\odd}^{(N)}}{2^{N+1} \prod_{j=1}^{N+1} \sin^2 \theta_j} (\cos \theta_{N+1} - \cos((2n'+1)\theta_{N+1})) \\
        &= \frac{\prod_{i=1}^{t}(1 - q^{-M_i})e_{\even}^{(N)}}{2^{N+1} \prod_{j=1}^{N+1} \sin^2 \theta_j}(1 - \cos((n+2)\theta_{N+1})) + \frac{\prod_{i=1}^{t}(1 - q^{-M_i}) e_{\odd}^{(N)}}{2^{N+1} \prod_{j=1}^{N+1} \sin^2 \theta_j} (\cos \theta_{N+1} - \cos((n+1)\theta_{N+1})) \\
        &= \frac{\prod_{i=1}^{t}(1 - q^{-M_i})(e_{\even}^{(N)} + e_{\odd}^{(N)} \cos \theta_{N+1})}{2^{N+1} \prod_{j=1}^{N+1} \sin^2 \theta_j} + \beta \sin(n \theta_{N+1} + \omega), 
    \end{align*}
    for some \(\beta, \omega \in \bR\) (use \eqref{eqn:sinadd}).
    Similarly, for \(n = 2n' + 1\) we have
    \begin{align*}
        &\sum_{k=0}^{n} \alpha_{N,k} \frac{\sin ((n+1-k) \theta_{N+1})}{\sin\theta_{N+1}} \\
        &= \sum_{k=0}^{n'} \alpha_{N,2k} \frac{\sin((n+1-2k) \theta_{N+1})}{\sin\theta_{N+1}} + \sum_{k=0}^{n'} \alpha_{N,2k+1} \frac{\sin((n-2k') \theta_{N+1})}{\sin\theta_{N+1}} \\
        &= \frac{\prod_{i=1}^{t}(1 - q^{-M_i}) e_{\even}^{(N)}}{2^N \prod_{j=1}^{N} \sin^2 \theta_j} \sum_{k=0}^{n'} \frac{\sin((2n'+2-2k)\theta_{N+1})}{\sin\theta_{N+1}} + \frac{\prod_{i=1}^{t}(1 - q^{-M_i}) e_{\odd}^{(N)}}{2^N \prod_{j=1}^{N} \sin^2 \theta_j} \sum_{k=0}^{n'} \frac{\sin((2n'+1-2k)\theta_{N+1})}{\sin\theta_{N+1}} \\
        &= \frac{\prod_{i=1}^{t}(1 - q^{-M_i}) e_{\even}^{(N)}}{2^{N+1} \prod_{j=1}^{N+1} \sin^2 \theta_j}(\cos\theta_{N+1} - \cos((2n'+3)\theta_{N+1})) + \frac{\prod_{i=1}^{t}(1 - q^{-M_i}) e_{\odd}^{(N)}}{2^{N+1} \prod_{j=1}^{N+1} \sin^2 \theta_j} (1 - \cos((2n'+2)\theta_{N+1})) \\
        &= \frac{\prod_{i=1}^{t}(1 - q^{-M_i})e_{\even}^{(N)}}{2^{N+1} \prod_{j=1}^{N+1} \sin^2 \theta_j}(\cos\theta_{N+1} - \cos((n+2)\theta_{N+1})) + \frac{\prod_{i=1}^{t}(1 - q^{-M_i}) e_{\odd}^{(N)}}{2^{N+1} \prod_{j=1}^{N+1} \sin^2 \theta_j} (1 - \cos((n+1)\theta_{N+1})) \\
        &= \frac{\prod_{i=1}^{t}(1 - q^{-M_i})(e_{\even}^{(N)} \cos\theta_{N+1} + e_{\odd}^{(N)})}{2^{N+1} \prod_{j=1}^{N+1} \sin^2 \theta_j} + \beta \sin(n \theta_{N+1} + \omega), 
    \end{align*}
    for some \(\beta, \omega \in \bR\).
    Now \eqref{eqn:alphaN} for \(N+1\) follows from
    \[
        e_\even^{(N)} + e_\odd^{(N)} \cos \theta_{N+1} = e_\even^{(N+1)}, \quad e_\even^{(N)} \cos \theta_{N+1} + e_\odd^{(N)} = e_\odd^{(N+1)}.
    \]

    Now, let's do induction on \(t\). By the induction hypothesis, we have
    \begin{align*}
        \cL_{t+1,N}^{\lambda}(u;m) &= (1 - q^{-M_{t+1}}u^{2M_{t+1}})\cL_{t,N}^{\lambda}(u;m) \\
        &=(1 - q^{-M_{t+1}}u^{2M_{t+1}}) \sum_{n \ge 0} b_{t,N}^\lambda(n;m) u^n
    \end{align*}
    so for \(n \ge 2M_1 + \cdots + 2M_t + 2M_{t+1}\), we have
    \begin{align*}
        b_{t+1,N}^\lambda(n;m) &= b_{t,N}^\lambda(n;m) - q^{-M_{t+1}} b_{t,N}^\lambda(n-2M_{t+1};m) \\
        &= \alpha_{t, N, n} - q^{-M_{t+1}} \alpha_{t, N, n-2M_{t+1}}\\
        &\quad + \sum_{j=1}^{N} \left(\beta_{j,t,N} \sin(n\theta_j + \omega_{j,t,N}) - q^{-M_{t+1}} \beta_{j,t,N} \sin((n-2M_{t+1})\theta_j + \omega_{j,t,N})\right).
    \end{align*}
    Using \eqref{eqn:sinadd}, we can combine the sine terms into a single sine term \(\beta_{j,t+1,N} \sin(n\theta_j + \omega_{j,t+1,N})\) for some \(\beta_{j,t+1,N}, \omega_{j,t+1,N} \in \bR\).
    Since \(n\) and \(n - 2M_{t+1}\) have the same parity, we have \(\alpha_{t,N,n} = \alpha_{t,N,n-2M_{t+1}}\) and obtain
    \[
    \alpha_{t+1,N,n} = (1 - q^{-M_{t+1}}) \alpha_{t,N,n} = \frac{\prod_{i=1}^{t+1} (1 - q^{-M_i})}{2^{N} \prod_{j=1}^{N} \sin^2 \theta_j} \times\begin{cases}
        e_\even^{(N)} & \text{if }2 \mid n \\
        e_\odd^{(N)} & \text{if }2 \nmid n.
    \end{cases}
    \]

    For even \(M\), we have an extra factor \(\frac{1}{1-q^{-1/2}u}\) in \(\cL^\lambda(q^{-\frac{1}{2}}u;m)\). From the previous computations for odd \(M\), we have
    \begin{align*}
        \cL^\lambda(q^{-\frac{1}{2}}u;m) &= \frac{1}{1 - q^{-\frac{1}{2}}u} \sum_{n \ge 1} b_{M'}^\lambda(n;m) u^n = \sum_{n \ge 0} \left(\sum_{k=0}^{n} q^{-\frac{n-k}{2}} b^\lambda_{M'}(k;m)\right) u^n = \sum_{n\ge 0} q^{-\frac{n}{2}} b^\lambda(n;m)u^n
    \end{align*}
    and
    \begin{equation}
        \label{eqn:blambdasumevenM}
        b^\lambda(n;m) = \sum_{k=0}^{n} q^{\frac{k}{2}} b^\lambda_{M'}(k;m) = \sum_{k=0}^{n} q^{\frac{k}{2}} \alpha_k + \sum_{j=1}^{M'} \beta_j \sum_{k=0}^{n} q^{\frac{k}{2}} \sin(k\theta_j + \omega_j)
    \end{equation}
    for \(\beta_j = \beta_{j,r,M'}\) and \(\omega_j = \omega_{j,r,M'}\).
    For the first sum, one can use \eqref{eqn:alphaN}.
    We consider two cases when \(n\) is even or odd separately.
    For even \(n\),
    \begin{align*}
        \sum_{k=0}^{n} q^{\frac{k}{2}} \alpha_k &= \frac{\prod_{i=1}^{r}(1 - q^{-M_i})e_{\even}^{(M')}}{2^{M'} \prod_{j=1}^{M'} \sin^2\theta_j}\sum_{k=0}^{\frac{n}{2}} q^k + \frac{\prod_{i=1}^{r}(1 - q^{-M_i})e_{\odd}^{(M')}}{2^{M'} \prod_{j=1}^{M'} \sin^2\theta_j}\sum_{k=0}^{\frac{n}{2}-1} q^{\frac{2k+1}{2}} \\
        &=\frac{\prod_{i=1}^{t}(1 - q^{-M_i})}{2^{M'} \prod_{j=1}^{M'} \sin^2\theta_j} \left(e_{\even}^{(M')} \cdot \frac{q^{\frac{n}{2}+1}-1}{q - 1} + e_{\odd}^{(M')} \cdot \frac{q^{\frac{n+1}{2}}-q^{\frac{1}{2}}}{q - 1} \right)
    \end{align*}
    and similarly for odd \(n\),
    \begin{align*}
        \sum_{k=0}^{n} q^{\frac{k}{2}} \alpha_k &= \frac{\prod_{i=1}^{r}(1 - q^{-M_i})e_{\even}^{(M')}}{2^{M'} \prod_{j=1}^{M'} \sin^2\theta_j}\sum_{k=0}^{\frac{n-1}{2}} q^k + \frac{\prod_{i=1}^{r}(1 - q^{-M_i})e_{\odd}^{(M')}}{2^{M'} \prod_{j=1}^{M'} \sin^2\theta_j}\sum_{k=0}^{\frac{n-1}{2}} q^{\frac{2k+1}{2}} \\
        &=\frac{\prod_{i=1}^{r}(1 - q^{-M_i})}{2^{M'} \prod_{j=1}^{M'} \sin^2\theta_j} \left(e_{\even}^{(M')} \cdot \frac{q^{\frac{n+1}{2}}-1}{q - 1} + e_{\odd}^{(M')} \cdot \frac{q^{\frac{n}{2}+1}-q^{\frac{1}{2}}}{q - 1} \right).
    \end{align*}
    The second summation can be simplified using \eqref{eqn:sinasum} and \eqref{eqn:sinadd} as
    \[
        C' + q^{\frac{n}{2}}\sum_{j=1}^{M'} \beta_j' \sin(n\theta_j + \omega_j')
    \]
    where
    \[
        \beta_j' = \beta_j \cdot \frac{1}{\sqrt{1 - 2 q^{-\frac{1}{2}} \cos\theta_j + q^{-1}}}
    \]
    and some \(\omega_j' \in \bR\),
    \begin{align*}
        C' = \sum_{j=1}^{M'} \beta_j \frac{\sin \omega_j - q^{\frac{1}{2}} \sin(\omega_j - \theta_j)}{1 - 2 q^{\frac{1}{2}}\cos\theta_j + q}.
    \end{align*}
    By combining the sums, we get \eqref{eqn:blambdaevenM} with
    \[
        C = C' - \frac{\prod_{i=1}^{r}(1 - q^{-M_i})}{2^{M'} \prod_{j=1}^{M'} \sin^2\theta_j}\left(
            \frac{e_\even^{(M')}}{q - 1} + \frac{e_\odd^{(M')} q^{\frac{1}{2}}}{q - 1}\right).
    \]
\end{proof}

\begin{corollary}
\label{cor:mainlambda}
    Assume that \(\cL(u, \chi_m)\) satisfies GSH.
    Let
    \[
    C_m = \frac{\prod_{i=1}^{r}(1 - q^{-M_i})}{2^{M'} \prod_{j=1}^{M'} \sin^2\theta_j}.
    \]
    Let \(\dd \phi\) be the normalized Haar measure on \(\bT^{M'}\).
    \begin{enumerate}
        \item Assume \(M\) is odd. For \(\beta_1, \dots, \beta_{M'}\) in \eqref{eqn:blambda}, define \(g : \bR^{M'} \to \bR\) as \(g(\mathbf{y}) = \sum_{j=1}^{M'} \beta_j \sin(y_j)\) and \(G : \bR^{M'} \to \bR\) as \(G(\mathbf{y}) = \sum_{j=1}^{M'} \beta_j' \sin(y_j)\). Then
        \begin{align}
            \delta_+^{\lambda, \nc}(m) &= \frac{1}{2} + \frac{1}{2} \dd \phi(\{C_m e_{\even}^{(M')} > g(\mathbf{y}) > - C_m e_{\odd}^{(M')}\}) \label{eqn:deltalambdaoddnc} \\
            \delta_+^{\lambda}(m) &= \frac{1}{2} + \frac{1}{2} \dd \phi\left(\left\{C_m \left( \frac{q e_{\even}^{(M')} + q^{\frac{1}{2}} e_{\odd}^{(M')}}{q-1}  \right)> G(\mathbf{y}) > - C_m \left(\frac{q^{\frac{1}{2}} e_{\even}^{(M')} + q e_{\odd}^{(M')}}{q-1}\right)\right\}\right) \label{eqn:deltalambdaodd}
        \end{align}
        where both are strictly greater than \(\frac{1}{2}\).
        \item Assume \(M\) is even.  For \(\beta_1, \dots, \beta_{M'}\) in \eqref{eqn:blambdaevenM}, define \(g\) and \(G\) similarly as above. Then
        \begin{align}
            &\delta_+^{\lambda, \nc}(m) \nonumber \\
            &= \frac{1}{2} + \frac{1}{2} \dd \phi\left(\left\{C_m \left( \frac{q e_{\even}^{(M')} + q^{\frac{1}{2}} e_{\odd}^{(M')}}{q-1} \right)> g(\mathbf{y}) > - C_m \left(\frac{q^{\frac{1}{2}} e_{\even}^{(M')} + q e_{\odd}^{(M')}}{q-1}\right)\right\}\right) \label{eqn:deltalambdaevennc}\\
            &\delta_+^{\lambda}(m) \nonumber \\
            &= \frac{1}{2} + \frac{1}{2} \dd \phi\left(\left\{C_m \left( \frac{(q^2 + q) e_{\even}^{(M')} + 2 q^{\frac{3}{2}} e_{\odd}^{(M')}}{(q-1)^{2}} \right)> G(\mathbf{y}) > - C_m \left(\frac{2q^{\frac{3}{2}} e_{\even}^{(M')} + (q^2 +q) e_{\odd}^{(M')}}{(q-1)^{2}}\right)\right\}\right) \label{eqn:deltalambdaeven}
        \end{align}
    \end{enumerate}
\end{corollary}
\begin{proof}
    Let \(\delta_{+,\even}^{\lambda, \nc}(m), \delta_{+,\odd}^{\lambda, \nc}(m)\) be the non-cumulative \(\lambda\)-densities for even and odd numbers, respectively, i.e.
    \begin{align*}
        \delta_{+,\even}^{\lambda, \nc}(m) := \lim_{n \to \infty}\frac{\#\{1 \le k \le n : b^\lambda(k; m) > 0, \,2 \mid k\}}{n/2}, \\
        \delta_{+,\odd}^{\lambda, \nc}(m) := \lim_{n \to \infty}\frac{\#\{1 \le k \le n : b^\lambda(k; m) > 0, \,2 \nmid k\}}{n/2}, \\
    \end{align*}
    so that \(\delta_+^{\lambda, \nc} = \frac{1}{2} (\delta_{+,\even}^{\lambda, \nc} + \delta_{+,\odd}^{\lambda, \nc})\).
    Assume \(M\) is odd and let \(M' = \frac{M-1}{2}\).
    Since \(\theta_1, \dots, \theta_{M'}\) are \(\bQ\)-linearly independent, we can apply Corollary \ref{cor:sindensity} to \eqref{eqn:blambda} to compute these densities as
    \begin{align*}
        \delta_{+,\even}^{\lambda, \nc}(m) &= \dd \phi(\{\mathbf{y} \in \bT^{M'} : g(\mathbf{y}) > - C_m \cdot e_{\even}^{(M')}\}) \\
        \delta_{+,\odd}^{\lambda, \nc}(m) &= \dd \phi(\{\mathbf{y} \in \bT^{M'} : g(\mathbf{y}) > - C_m \cdot e_{\odd}^{(M')}\}).
    \end{align*}
    Since \(g(-\mathbf{y}) = -g(\mathbf{y})\) and the hypersurface \(\{\mathbf{y} \in \bT^{M'} : g(\mathbf{y}) = 0\}\) has measure zero, both \(\{g(\mathbf{y}) > 0\}\) and \(\{g(\mathbf{y}) < 0\}\) have measure \(\frac{1}{2}\).
    Also, we have
    \begin{align*}
        e_{\even}^{(M')} &= \frac{(e_{\even}^{(M')} + e_{\odd}^{(M')}) + (e_{\even}^{(M')} - e_{\odd}^{(M')})}{2} = \frac{1}{2}\prod_{j=1}^{M'} (1 + \cos \theta_j) + \frac{1}{2}\prod_{j=1}^{M'} (1 - \cos \theta_j) > 0 \\
        e_{\even}^{(M')} + e_{\odd}^{(M')} &= \sum_{k=0}^{M'} e_k^{(M')} = \prod_{j=1}^{M'} (1 + \cos \theta_j) > 0.
    \end{align*}
    Using these, we can rewrite the densities as
    \begin{align*}
        \delta_{+,\even}^{\lambda, \nc}(m) &= \frac{1}{2} + \dd \phi(\{\mathbf{y} \in \bT^{M'} : C_m \cdot e_{\even}^{(M')}> g(\mathbf{y}) > 0\}) \\
        &= \frac{1}{2} + \dd \phi(\{\mathbf{y} \in \bT^{M'} : 0> g(\mathbf{y}) > - C_m \cdot e_{\even}^{(M')}\}) \\
        \delta_{+,\odd}^{\lambda, \nc}(m) &= \frac{1}{2} + \begin{cases}
            \dd\phi(\{\mathbf{y} \in \bR^{M'}: 0 > g(\mathbf{y}) > -C_m \cdot e_{\odd}^{(M')}\}) & \text{if }e_{\odd}^{(M')} > 0 \\
            -\dd\phi(\{\mathbf{y} \in \bR^{M'}: 0 > g(\mathbf{y}) > C_m \cdot e_{\odd}^{(M')}\}) & \text{if }e_{\odd}^{(M')} < 0
        \end{cases} \\
        &= \frac{1}{2} + \begin{cases}
            \dd\phi(\{\mathbf{y} \in \bR^{M'}: C_m \cdot e_{\odd}^{(M')} > g(\mathbf{y}) > 0 \}) & \text{if }e_{\odd}^{(M')} > 0 \\
            -\dd\phi(\{\mathbf{y} \in \bR^{M'}: 0 > g(\mathbf{y}) > C_m \cdot e_{\odd}^{(M')}\}) & \text{if }e_{\odd}^{(M')} < 0.
        \end{cases}
    \end{align*}
    (The last equality follows from the symmetry of \(g\) as before.)
    In either case of the sign of \(e_{\odd}^{(M')}\), we have
    \[
        \delta_{+}^{\lambda, \nc}(m) = \frac{\delta_{+,\even}^{\lambda, \nc}(m) + \delta_{+,\odd}^{\lambda, \nc}(m)}{2} = \frac{1}{2} + \frac{1}{2} \dd \phi(\{C_m \cdot e_{\odd}^{(M')} > g(\mathbf{y}) > -C_m \cdot e_{\even}^{(M')}\}) > \frac{1}{2}.
    \]
    For even \(M\), we can use \eqref{eqn:alphanevenM} instead.
    Note that
    \[
    \frac{q e_{\even}^{(M')} + q^{\frac{1}{2}} e_{\odd}^{(M')}}{q-1} + \frac{q^{\frac{1}{2}} e_{\even}^{(M')} + q e_{\odd}^{(M')}}{q-1} = \frac{e_{\even}^{(M')} + e_{\odd}^{(M')}}{1 - q^{-\frac{1}{2}}} > 0.
    \]

    For cumulative densities, when \(M\) is odd, one needs to study the sign of
    \begin{equation}
        \label{eqn:Blambdasum}
        B^\lambda(n;m) = \sum_{k=1}^{n} b^\lambda(k;m) = \sum_{k=1}^{n} q^{\frac{k}{2}} \alpha_k + \sum_{j=1}^{M'} \beta_j \sum_{k=1}^{n} q^{\frac{k}{2}} \sin(k\theta_j + \omega_j).
    \end{equation}
    The above summation is almost identical to the summation \eqref{eqn:blambdasumevenM}, except that the first summation starts from \(k=1\) instead of \(k=0\). 
    Similar computation shows that the constant term of \eqref{eqn:Blambdasum} that contributes to the cumulative density is the same as \eqref{eqn:alphanevenM} and this proves \eqref{eqn:deltalambdaodd}.
    For even \(M\), the bias for the cumulative density comes from the summation
    \[
    \sum_{k=1}^{n} q^{\frac{k}{2}} \alpha_k' = -C_m \left(\frac{(q^2 + q) e_{\even}^{(M')} + 2 q^{\frac{3}{2}} e_{\odd}^{(M')}}{(q-1)^{2}} \right) + q^{\frac{n}{2}} \times \begin{cases}
        C_m \left(\frac{(q^2 + q) e_{\even}^{(M')} + 2 q^{\frac{3}{2}} e_{\odd}^{(M')}}{(q-1)^{2}} \right) & \text{if } 2 \mid n \\
        C_m \left(\frac{2q^{\frac{3}{2}} e_{\even}^{(M')} + (q^2 +q) e_{\odd}^{(M')}}{(q-1)^{2}}\right) & \text{if } 2 \nmid n
    \end{cases}
    \]
    which can be proved using \eqref{eqn:alphanevenM}, and this implies \eqref{eqn:deltalambdaeven}.
\end{proof}

One may ask whether the bias disappears as \(\deg m \to \infty\). We can relate the density to the central \(L\)-value \(L(\frac{1}{2}, \chi_m) = \cL(q^{-\frac{1}{2}}, \chi_m)\).

\begin{proposition}
    \label{prop:gapLfunc}
    \[
     \frac{\prod_{i=1}^{r}(1 - q^{-M_i})}{2^{M'} \prod_{j=1}^{M'} \sin^2 \theta_j} (e_\even^{(M')} + e_\odd^{(M')}) = \frac{\prod_{i=1}^{r}(1 - q^{-M_i})}{\cL(q^{-\frac{1}{2}}, \chi_m)} \times \begin{cases} 1 & \text{if }2 \nmid M \\ 1 - q^{-\frac{1}{2}} & \text{if } 2 \mid M \end{cases}
    \]
\end{proposition}
\begin{proof}
    By Theorem \ref{thm:Lfuncquad}, we have
    \begin{align*}
        \cL(q^{-\frac{1}{2}}, \chi_m) &= \begin{cases} \prod_{j=1}^{M'} (2 - 2 \cos \theta_j) & \text{if } 2 \nmid M \\ (1 - q^{-\frac{1}{2}}) \prod_{j=1}^{M'} (2 - 2 \cos \theta_j) & \text{if } 2 \mid M \end{cases}
    \end{align*}
    and use the identity \(\frac{1 + \cos x}{2 \sin^2 x} = \frac{1}{2 - 2 \cos x}\) to conclude that
    \begin{align*}
        \frac{\prod_{i=1}^{r}(1 - q^{-M_i})}{2^{M'} \prod_{j=1}^{M'} \sin^2 \theta_j} (e_{\even}^{(M')} + e_{\odd}^{(M')}) &= \prod_{i=1}^{r} (1 - q^{-M_i}) \cdot \frac{1 + \cos \theta_1}{2 \sin^2 \theta_1} \cdots \frac{1 + \cos \theta_{M'}}{2 \sin^2 \theta_{M'}} \\
        &= \prod_{i=1}^{r}(1 - q^{-M_i}) \prod_{j=1}^{M'} \frac{1}{2 - 2 \cos \theta_j} \\
        &= \frac{\prod_{i=1}^{r}(1 - q^{-M_i})}{\cL(q^{-\frac{1}{2}}, \chi_m)} \times \begin{cases} 1 & \text{if }2 \nmid M \\ 1 - q^{-\frac{1}{2}} & \text{if } 2 \mid M \end{cases}
    \end{align*}
\end{proof}

From Corollary \ref{cor:mainlambda}, it is clear that the bias becomes smaller (i.e. densities get close to \(\frac{1}{2}\)) if and only if \(\cL(q^{-\frac{1}{2}}, \chi_m) \) gets larger.
Andrade and Keating \cite{andrade2012mean} proved an asymptotic formula of the average of the central \(L\)-values (when \(q \equiv 1 \pmod{4}\)) of fixed degree, which grows linearly in \(\deg m\).
In other words, for any given \(N\), there exists a monic square-free polynomial of degree \(\ge N\) with large central \(L\)-values.
By the work of Bailleul, Devin, Keliher, and Li \cite[Theorem 1.1]{bailleul2024exceptional}, $\cL(u, \chi_m)$ satisfies GSH for almost all monic square-free polynomials $m$ with sufficiently large degree.
Combining these results with Proposition \ref{prop:gapLfunc}, we have the following corollary.
\begin{corollary}
    \label{cor:gapLfunc}
    Assume \(q \equiv 1 \pmod{4}\).
    For any \(\epsilon > 0\), there exists a monic square-free polynomial \(m\) such that \(\delta_+^{\lambda, \nc}(m) < \frac{1}{2} + \epsilon\) and \(\delta_{+}^{\lambda}(m) < \frac{1}{2} + \epsilon\).
\end{corollary}

%% file: src/5examples.tex
\section{Examples}
\label{sec:examples}

We give several examples of the densities \(\delta_{\pm}^{\bullet}(m)\) when \(m\) has a small degree.

\subsection{When \texorpdfstring{\(\deg m = 1\)}{deg m = 1}}

Let \(m(T) = T - a\) be a monic degree 1 polynomial in \(A = \bF_q[T]\).
Then the corresponding Dirichlet $L$-function is $\cL(u, \chi_m) = 1$, so
\[
\cL^\mu(u; m) = \frac{1}{\cL(u, \chi_m)} = 1, \quad \cL^\lambda(u; m) = \frac{1 - u^2}{1 - qu^2} = 1 + \sum_{n \ge 1} (q - 1) \cdot q^{n-1} \cdot u^{2n}.
\]
From the computation, we get
\[
    b^\mu(n;T) = 0, \quad b^\lambda(n; T) = \begin{cases}
        (q - 1) \cdot q^{\frac{n}{2} - 1} & 2 \mid n \\
        0 & 2 \nmid n
    \end{cases}
\]
for all $n \ge 1$.
Especially, \(\mu \chi_m\) is always on tie for all degree, while \(\lambda \chi_m\) is biased towards \(+1\) over \(-1\) for even degree.
This gives non-cumulative densities
\begin{align*}
    \delta_+^{\mu, \nc}(m) = \delta_-^{\mu, \nc}(m) = 0,\,\delta_0^{\mu, \nc}(m) = 1 \\
    \delta_+^{\lambda, \nc}(m) = \delta_-^{\lambda, \nc}(m) = \frac{1}{2},\,\delta_0^{\lambda, \nc}(m) = 1.
\end{align*}
Also, the above formula implies that \(B^{\lambda}(n; T) > 0\) for all \(n > 1\), so the cumulative densities are
\begin{align*}
    \delta_+^{\mu}(m) = \delta_-^{\mu}(m) = 0,\,\delta_0^{\mu}(m) = 1 \\
    \delta_+^{\lambda}(m) = 1,\,\delta_-^{\lambda}(m) = \delta_0^{\lambda}(m) = 0.
\end{align*}

\subsection{When \texorpdfstring{\(\deg m = 2\)}{deg m = 2}}

Let \(m\) be a monic non-square polynomial.
By Theorem~\ref{thm:Lfuncquad}, \(\cL(u, \chi_m) = 1 - u\) and
\begin{align*}
    \cL^\mu(u; m) &= \frac{1}{1 - u} = \sum_{n \ge 0} u^n, \\
    \cL^\lambda(u; m) &= \frac{1 - u^{4}}{1 - qu^2} \cdot \frac{1}{1-u} \\
    &= (1 + u + u^2 + u^3) (1 + qu^2 + q^2 u^4 + \cdots) = 1 + u + \sum_{k \ge 1} (q + 1) \cdot q^{k-1} \cdot (u^{2k} + u^{2k + 1}).
\end{align*}
From the computation, we get
\[
    b^\mu(n; m) = 1, \quad b^\lambda(n; m) = \begin{cases}
        1 & n = 1 \\
        (q + 1) \cdot q^{\lfloor \frac{n}{2} \rfloor - 1} & n \ge 2
    \end{cases}
\]
for all \( n \ge 1\). Especially, it is always biased towards \(+1\) over \(-1\), hence \(\delta_+^{\mu, \nc}(m) = \delta_+^{\lambda, \nc}(m) = \delta_+^{\mu}(m) = \delta_+^{\lambda}(m) = 1\) while all other densities are zero.

\subsection{When \texorpdfstring{\(\deg m = 3\)}{deg m = 3}}

The corresponding Dirichlet \(L\)-function has degree 2, and it can be factored as \[\cL(u, \chi_m) = (1 - \sqrt{q} e^{i \theta}u)(1 - \sqrt{q} e^{-i\theta}u)\]
for some \(\theta \in (0, \pi)\).
Following the computations in the proof of Theorem~\ref{thm:blambdabias}, we can compute \(\cL^\mu\) and \(\cL^\lambda\) as
\begin{align*}
    \cL^\mu(u;m) &= \frac{1}{(1 - \sqrt{q} e^{i\theta} u) (1 - \sqrt{q} e^{-i\theta}u)} = \left(\sum_{n \ge 0} q^{\frac{n}{2}} e^{in\theta} u^n\right) \left(\sum_{n \ge 0} q^{\frac{n}{2}} e^{-in\theta} u^n\right) \\
    &= \sum_{n \ge 0} \frac{q^{\frac{n}{2}} \sin ((n+1) \theta)}{\sin \theta} u^n \\
    \cL^\lambda(u;m) &= \frac{1 - u^6}{1 - qu^2} \cL^\mu(u;m) \\
    &= (1 - u^6) \left[ \sum_{k \ge 0} \left(\frac{q^{k} (1 - \cos((2k+2)\theta))}{2\sin^2 \theta}\right) u^{2k} + \left(\frac{q^{k + \frac{1}{2}}(\cos \theta - \cos((2k+3)\theta))}{2\sin^2 \theta}\right) u^{2k+1}  \right]
\end{align*}
and it gives
\begin{align*}
    b^\mu(n;m) &= \frac{q^n \sin((n+1)\theta)}{\sin\theta}, \\
    b^\lambda(n;m) &= \begin{cases}
        \dfrac{q^{\frac{n}{2}}(1 - q^{-3} - \cos((n+2)\theta) + q^{-3} \cos((n-4)\theta))}{2\sin^2\theta} & \text{if }2 \mid n \\
        \dfrac{q^{\frac{n}{2}}((1 - q^{-3}) \cos \theta - \cos((n+2)\theta) + q^{-3} \cos((n-4)\theta))}{2\sin^2\theta} & \text{if }2 \nmid n
    \end{cases}
\end{align*}
for \(n \ge 6\).

When \(2 \mid n\), \(b^\lambda(n;m) > 0\) is equivalent to
\begin{align*}
    &1 - q^{-3} - \cos((n + 2)\theta) + q^{-3} \cos((n-4)\theta) > 0 \\
    &\Leftrightarrow 1 - q^{-3} - [\cos((n-4)\theta) \cos(6\theta) - \sin((n-4)\theta) \sin(6\theta)] + q^{-3} \cos((n-4)\theta) > 0 \\
    &\Leftrightarrow 1 - q^{-3} + \sqrt{(-\cos (6\theta) + q^{-3})^{2} + \sin^2(6\theta)} \cdot \sin(n \theta + \omega) > 0 \\
    &\Leftrightarrow \sin (n \theta + \omega) > -\frac{1 - q^{-3}}{\sqrt{1 - 2q^{-3} \cos(6\theta) + q^{-6}}}
\end{align*}
where \(\omega \in \bR\) is an angle only depends on \(\theta\) and \(q\).
Then the density of such (even) \(n\) would be
\[
\frac{1}{\pi}\left[\frac{\pi}{2} + \sin^{-1}\left(\frac{1 - q^{-3}}{\sqrt{1 - 2q^{-3} \cos(6\theta) + q^{-6}}}\right)\right] = \frac{1}{2} + \frac{1}{\pi}\sin^{-1}\left(\frac{1 - q^{-3}}{\sqrt{1 - 2q^{-3} \cos(6\theta) + q^{-6}}}\right).
\]
Similarly, for odd \(n\) we have \(b^\lambda(n;m) > 0\) if and only if
\[
\sin(n\theta + \omega') > -\frac{(1 - q^{-3})\cos \theta}{\sqrt{1 - 2q^{-3} \cos(6\theta) + q^{-6}}}
\]
for some \(\omega' \in \bR\) only depends on \(\theta, q\), and the density of such \(n\) would be
\[
\frac{1}{2} + \frac{1}{\pi}\sin^{-1}\left(\frac{(1 - q^{-3})\cos\theta}{\sqrt{1 - 2q^{-3} \cos(6\theta) + q^{-6}}}\right).
\]
Hence \(\delta_+^{\lambda, \nc}(m)\) would be the average of these two numbers, which is
\[
    \delta_+^{\lambda, \nc}(m) = \frac{1}{2} + \frac{1}{2\pi} \left(\sin^{-1} (a) + \sin^{-1} (a \cos \theta)\right)
\]
where
\[
1 > a = \frac{1 - q^{-3}}{\sqrt{1 - 2q^{-3} \cos(6\theta) + q^{-6}}} > 0.
\]
By applying~\eqref{eqn:sinasum} again, one can compute \(B^\lambda(n; m)\) as
\begin{align*}
    2 \sin^2\theta \cdot q^{-\frac{n}{2}}B^\lambda(n;m) &= \sqrt{\frac{1 - 2q^{-3} \cos(6\theta) + q^{-6}}{1 - 2q^{-\frac{1}{2}}\cos\theta + q^{-1}}} \sin(n\theta + \omega') + \frac{1 - q^{-3}}{q - 1} \times \begin{cases}
        q + q^{\frac{1}{2}} \cos \theta & \text{if } 2 \mid n \\
        q \cos \theta + q^{\frac{1}{2}} & \text{if } 2 \nmid n
    \end{cases} \\
    &\quad+ o(1),
\end{align*}
and Corollary~\ref{cor:sindensity} gives the cumulative density
\[
\delta_+^{\lambda}(m) = \frac{1}{2} + \frac{\sin^{-1}(a_{\even}) + \sin^{-1}(a_{\odd})}{2 \pi}
\]
where
\begin{align*}
    a_\even &= \sqrt{\frac{1 - 2q^{-1/2}\cos\theta + q^{-1}}{1 - 2q^{-3} \cos(6\theta) + q^{-6}}} \cdot \frac{(1- q^{-3})(q + q^{\frac{1}{2}}\cos\theta)}{q - 1} \\
    a_\odd &= \sqrt{\frac{1 - 2q^{-1/2}\cos\theta + q^{-1}}{1 - 2q^{-3} \cos(6\theta) + q^{-6}}} \cdot \frac{(1- q^{-3})(q \cos \theta + q^{\frac{1}{2}})}{q - 1} 
\end{align*}
and we extend \(\sin^{-1}(x)\) to \(\bR\) by defining \(\sin^{-1}(x) = \frac{\pi}{2}\) for \(x \ge 1\) and \(\sin^{-1}(x) = -\frac{\pi}{2}\) for \(x \le -1\).

For example, the Dirichlet \(L\)-function for \(q = 5\) and \(m = T^3 + T + 4\) is
\[
    \cL(u, \chi_m) = 1 + 3u + 5u^2 = \left(1 - \frac{-3 + \sqrt{11}i}{2} u\right) \left(1 - \frac{-3 - \sqrt{11}i}{2} u\right) = (1 - \sqrt{5} e^{i\theta})(1 - \sqrt{5} e^{-i\theta})
\]
where \(\theta = \cos^{-1}\left(-\frac{3}{\sqrt{20}}\right) \in (0, \pi)\).
Since \(\tan\theta = -\frac{\sqrt{11}}{3}\) is not on the Calcut's list~\cite[p. 17]{calcut2006rationality}, we have \(\theta \not\in \bQ\pi\).
In this case, \(a = \frac{31\sqrt{3}}{54},\,a_\even = \frac{217\sqrt{3}}{144\sqrt{5}} > 1\) and \(a_\odd = -\frac{31\sqrt{3}}{144}\), so the densities are
\begin{align*}
    \delta_+^{\lambda, \nc}(T^3 + T + 4) &= \frac{1}{2} + \frac{1}{2\pi} \left(\sin^{-1}\left(\frac{31\sqrt{3}}{54}\right) - \sin^{-1}\left(\frac{31\sqrt{15}}{180}\right)\right) \approx 0.6168\dots \\
    \delta_+^{\lambda}(T^3 + T + 4) &= \frac{1}{2} + \frac{1}{2\pi} \left(\frac{\pi}{2} - \sin^{-1}\left(\frac{31\sqrt{3}}{144}\right)\right)\approx 0.6892\dots
\end{align*}
We compare this with numerical computations.
From the \(L\)-function, one can prove the following recurrence relation satisfied by \(b_n = b^\lambda(n; T^3 + T + 4)\): 
\[
b_n = -3b_{n-1} + 15 b_{n-3} + 25b_{n-4}
\]
which holds for \(n \ge 7\), and with the first few values are \((b_1, b_2, b_3, b_4, b_5, b_6) = (-3, 9, -12, 16, 12, 8)\), we can compute further values efficiently\footnote{Using fast matrix exponentiation, \(b_n\) can be computed in \(O(\log n)\) time.}.
Table~\ref{tab:T3T4} shows that the numerical computations converge to the above densities as \(n\) increases.

\begin{table}
    \centering
    \begin{tabular}{c|cccc}
    \toprule
    \(n\) & \(\delta_+^{\lambda, \nc}(m)\) & \(\delta_+^{\lambda}(m)\) & \(\delta_+^{\mu, \nc}(m)\) & \(\delta_+^{\mu}(m)\) \\ 
    \midrule
    10    & 0.6000 & 0.6000 & 0.5000 & 0.4000 \\
    100   & 0.6200 & 0.6800 & 0.4900 & 0.4900 \\
    1000  & 0.6160 & 0.6900 & 0.4990 & 0.5000 \\
    10000 & 0.6168 & 0.6891 & 0.4999 & 0.4999 \\
    \bottomrule
    \end{tabular}
    \caption{Numerical computations of \(\lambda\) and \(\mu\)-densities  for \(m = T^3 + T + 4 \in \bF_5[T]\) using polynomials of degree up to \(10, 100, 1000, 10000\).}
    \label{tab:T3T4}
\end{table}

Note that \(a_\even > 1\) indicates 100\% of the even degree polynomials \(f\) satisfy \(\lambda(f) \chi_m(f) = +1\).
In fact, \(b_n' = b_{2n}\) satisfies the recurrence relation \(b_n' = 4b_{n-1}' - 20 b_{n-2}' + 125 b_{n-3}'\), and using characteristic polynomial with the first few values \((b_1', b_2', b_3') = (9, 16, 8)\), we get a general formula \(b_n' = \frac{2 \cdot 5^{n+1}}{11} (1 - \cos((n+1)\theta))\) where \(\theta = -\tan^{-1} \sqrt{99} \in (0, \pi)\). This shows \(b_n' > 0\) for all \(n \ge 1\).

As shown in~\cite{cha2008chebyshev}, for some \(q\) and \(m\), \(\cL(u, \chi_m)\) may not satisfy the GSH.
If we write \(\cL(u, \chi_m) = 1 - 2 q^{\frac{1}{2}} \cos \theta u + q u^2\), then \(2 q^{\frac{1}{2}} \cos\theta \in \bZ\) and \(\cos \theta, \sin\theta, \tan\theta\) have to be rational or at most quadratic irrational.
Again, Calcut~\cite[p. 17]{calcut2006rationality} provided a complete \emph{finite} list of angles \(\theta \in \bQ\pi \cap (0, \pi)\) where \(\tan \theta \) is at most a quadratic irrational, and some of them corresponds to \(L\)-functions of Dirichlet characters. 
In this case, the density can be different form the case where the GSH holds.
For example, let \(q = 3\) and \(m = T^3 - T + 1\), which is Example 5.1 of~\cite{cha2008chebyshev}.
The corresponding Dirichlet $L$-function of \(\chi_m = \chi_{T^3 - T + 1}\) is \[\cL(u, \chi_m) = 1 - 3 u + 3 u^2 = (1 - \sqrt{3} \zeta_{12}u)(1 - \sqrt{3} \zeta_{12}^{-1}u).\]
Then
\begin{align*}
    \cL^\mu(u;m) &= \left(\sum_{n \ge 0} 3^{n/2} \zeta_{12}^{n} u^n\right)\left( \sum_{n \ge 0} 3^{n/2} \zeta_{12}^{-n} u^{n}\right) = \sum_{n \ge 0} \frac{3^{\frac{(n+1)}{2}}(\zeta_{12}^{n+1} - \zeta_{12}^{-(n+1)})}{\sqrt{3}(\zeta_{12} - \zeta_{12}^{-1})} u^n \\
    &= \sum_{n \ge 0} 2 \cdot 3^{n/2} \sin \left(\frac{(n+1)\pi}{6}\right) u^n = \sum_{n \ge 0} b^\mu(n;m) u^n \\
    \cL^\lambda(u;m) &= \frac{1 - u^6}{1 - 3u^2} \cdot \frac{1}{1 - 3u + 3u^2} \\
    &= (1 - u^6) \left[\sum_{k \ge 0} 2 \cdot 3^k \cdot \left(\sin \frac{\pi}{6} + \sin \frac{3\pi}{6} + \cdots + \sin \frac{(2k+1)\pi}{6}\right) u^{2k} \right. \\
    &\quad+ \left. 2 \cdot 3^{k + \frac{1}{2}} \cdot \left(\sin \frac{2 \pi}{6} + \sin \frac{4\pi}{6} + \cdots + \sin \frac{(2k+2)\pi}{6}\right)u^{2k+1}\right] \\
    &=: (1 - u^6) \sum_{n \ge 0} c_n u^n \\
    &= \sum_{n \ge 0} (c_n - c_{n-6}) u^n
\end{align*}
where
\begin{align*}
    b^\lambda(n;m) = b^\lambda({12a + k};m) = \begin{cases}
        3^{6a} & k = 0 \\
        3^{6a + 1} & k = 1 \\
        2 \cdot 3^{6a + 1} & k = 2 \\
        3^{6a + 2} & k = 3, 4 \\
        0 & k = 5, 11 \\
        - 3^{6a + 3} & k = 6\\
        -3^{6a + 4} & k = 7 \\
        -2 \cdot 3^{6a + 4} & k = 8 \\
        -3^{6a + 5} & k = 9, 10
    \end{cases}, \quad 
    c_n = c_{12a + k} = \begin{cases}
        3^{6a} & k = 0 \\
        3^{6a + 1} & k = 1\\
        3^{6a + 2} & k = 2 \\
        2 \cdot 3^{6a + 2} & k = 3 \\
        4 \cdot 3^{6a + 2} & k = 4 \\
        2 \cdot 3^{6a + 3} & k = 5 \\
        3^{6a + 4} & k = 6, 7, 8 \\
        0 & k = 9, 10, 11
    \end{cases}
\end{align*}
(and we set \(c_n = 0\) for \(n < 0\)).
The formula for \(b^\lambda(n;m)\) gives \(\delta_+^{\mu, \nc}(m) = \delta_-^{\mu, \nc}(m) = \frac{5}{12}\) and \(\delta_0^{\mu, \nc}(m) = \frac{1}{6}\).
Also, \(B^\lambda(n;m) = \sum_{1 \le k \le n} b^\lambda(k;m)\) is positive / zero / negative if and only if \(n \equiv 1, \dots, 6 \pmod{12}\) / \(n \equiv 0, 6 \pmod{12}\) / \(n \equiv 7, \dots, 11\pmod{12}\), where we get\footnote{Note that the tie breaks if we include the degree 0 constant polynomial \(f(T) = 1\) in the counting, and the density becomes \(\frac{7}{12}, 0, \frac{5}{12}\).} \[\delta_+^{\mu, \nc}(m) = \delta_-^{\mu, \nc}(m) = \frac{5}{12},\, \delta_0^{\mu, \nc}(m) = \frac{1}{6}.\]

For $\lambda$, one can check that \(c_n - c_{n-6} > 0\) if \(n \equiv 0, 1, \dots, 8 \pmod{12}\) and \(c_n - c_{n-6} < 0\) otherwise, which gives \[\delta_+^{\lambda, \nc}(m) = \frac{3}{4},\,\delta_-^{\lambda, \nc}(m) = \frac{1}{4},\,\delta_0^{\lambda, \nc}(m) = 0.\]
For the cumulative density, we have
\begin{align*}
\frac{\cL^\lambda(u; m) - 1}{1 - u} &= \frac{(1 - u^6)(\sum_{n \ge 1} c_n u^n) - 1}{1 - u} \\
&=(1 + u + \cdots + u^5) \left(\sum_{n \ge 1} c_n u^n\right) - (1 + u + u^2 + \cdots) \\
&=c_1 u + (c_2 + c_1)u^2 + \cdots + (c_5 + \cdots + c_1) u^5 + \sum_{n \ge 6} (c_n + c_{n-1} + \cdots + c_{n-5} - 1) u^n
\end{align*}
where all the coefficients are positive, so the cumulative count is completely biased towards \(+1\) and \(\delta_+^{\lambda}(m) = 1\).

\subsection{Higher degree example}

Consider \(m = (T^2 + 1)(T^3 + 2T + 1) \in \bF_3[T]\), which has degree 5 and is not irreducible.
The corresponding Dirichlet \(L\)-function is
\begin{align*}
    \cL(u, \chi_m) &= 1 + u + 4u^2 + 3u^3 + 9u^4 \\
    &= (1 - u + 3u^2) (1 + 2u + 3u^2) \\
    &= (1 - \sqrt{3} e^{i\theta_1} u)(1 - \sqrt{3} e^{-i\theta_1} u)(1 - \sqrt{3} e^{i\theta_2} u)(1 - \sqrt{3} e^{-i\theta_2} u)
\end{align*}
where
\[
e^{i\theta_1} = \frac{1 + \sqrt{-11}}{2\sqrt{3}} \Leftrightarrow \cos\theta_1 = \frac{1}{2\sqrt{3}}, \quad e^{i\theta_2} = \frac{-1 + \sqrt{-2}}{\sqrt{3}} \Leftrightarrow \cos\theta_2 = -\frac{1}{\sqrt{3}}.
\] 
One can check that \(\pi, \theta_1, \theta_2\) are linearly independent over \(\bQ\), i.e. GSH holds for \(\cL(u, \chi_m)\).
Minimal polynomial of \(e^{i\theta_1}\) and \(e^{i\theta_2}\) are \(x^4 + \frac{5}{3}x^2 + 1\) and \(x^4 + \frac{2}{3} x^2 + 1\) respectively, which are not monic. Hence \(e^{i\theta_1}, e^{i\theta_2}\) cannot be roots of unity and \(\theta_1, \theta_2 \not\in \bQ \pi\).
To show the linear independence, it is enough to show that \(e^{i(m\theta_1 + n\theta_2)} = 1 \Leftrightarrow \left(\frac{1 + \sqrt{-11}}{2\sqrt{3}}\right)^m \left(\frac{-1 + \sqrt{-2}}{\sqrt{3}}\right)^n = 1\) implies \((m, n) = (0, 0)\), for \(m, n \in \bZ\). If the relation holds, then \((1 + \sqrt{-11})^m \in \bQ(\sqrt{3}, \sqrt{-2})\) and this implies \(m = 0\): if not, \(\theta_1 \not \in \bQ\pi\) implies that \(\Im((1 + \sqrt{-11})^m) \ne 0\) and \(\sqrt{-11} \in \bQ(\sqrt{3}, \sqrt{-2})\), a contradiction. Then \(\left(\frac{-1 + \sqrt{-2}}{\sqrt{3}}\right)^n = 1\) and \(\theta_2 \not \in \bQ\pi\) implies \(n = 0\).

The \(L\)-functions for (non-cumulative) \(\lambda\)-bias is (Corollary~\ref{cor:Llambda})
\[
    \cL^\lambda(u;m) = \frac{(1 - u^4)(1 - u^6)}{(1 - 3u^2)(1 - u + 3u^2)(1 + 2u + 3u^2)} 
\]
We have
\begin{align*}
    C_m = \frac{(1 - 3^{-2})(1 - 3^{-3})}{2^2 \cdot \sin^2 \theta_1 \cdot \sin^2 \theta_2} = \frac{104}{297}, \quad
    e_{\even}^{(M')} = 1 + \cos\theta_1 \cos\theta_2 = \frac{5}{6}, \quad
    e_{\odd}^{(M')} = \cos\theta_1 + \cos\theta_2 = -\frac{1}{2\sqrt{3}}
\end{align*}
and using partial fraction, we can write \(b_n := b^\lambda(n; m)\) as a trigonometric form.
Without denominator, we have
\begin{align*}
&\frac{1}{(1 - 3u^2)(1 - u + 3u^2)(1 + 2u + 3u^2)} \\
&= \frac{3}{44} \frac{5 - 3u}{1 - 3u^2} + \frac{8 - 15u}{33(1 -u + 3u^2)} + \frac{5 + 3u}{12(1 + 2u + 3u^2)} \\
&= \sum_{n \ge 0} \left[\frac{3}{88}((5 - \sqrt{3}) + (-1)^n (5 + \sqrt{3})) + \frac{2\sqrt{3}}{11} \sin(n\theta_1 + \omega_1') + \frac{\sqrt{3}}{4} \sin(n\theta_2 + \omega_2')\right] 3^{\frac{n}{2}} u^n
\end{align*}
and multiplying by \(1 - u^4\) and \(1 - u^6\) gives (apply \eqref{eqn:sinadd} twice)
\begin{align*}
    b_n &= 3^{\frac{n}{2}} \left(\alpha_n + \beta_{1} \sin(n\theta_1 + \omega_1) + \beta_{2} \sin(n\theta_2 + \omega_2) \right) \\
    \alpha_n &= \begin{cases}
        \frac{260}{891} & \text{if } 2 \mid n \\
        -\frac{52\sqrt{3}}{891} & \text{if } 2 \nmid n
    \end{cases} \\
    \beta_{1} &= \frac{2\sqrt{3}}{11} \prod_{i=1}^{2} \sqrt{1 - 2 \cdot 3^{-M_i} \cos (2M_i \theta_1) + 3^{-2M_i}} = \frac{40\sqrt{5}}{297} \\
    \beta_{2} &= \frac{\sqrt{3}}{4} \prod_{i=1}^{2} \sqrt{1 - 2 \cdot 3^{-M_i} \cos (2M_i \theta_2) + 3^{-2M_i}} = \frac{2\sqrt{38}}{27}
\end{align*}
for \(n \ge 10\). Similarly, the cumulative bias \(B_n = B^\lambda(n; m)\) can be computed as
\begin{align*}
    B_n &= 3^{\frac{n}{2}} \left(\alpha_n' + \beta_1' \sin(n\theta_1 + \omega_1') + \beta_2' \sin(n\theta_2 + \omega_2') + o(1)\right) \\
    \alpha_n' &= \begin{cases}
        \frac{104}{297} & \text{if } 2 \mid n \\
        \frac{52\sqrt{3}}{891} & \text{if } 2 \nmid n
    \end{cases} \\
    \beta_1' &= \beta_1 \cdot \frac{1}{\sqrt{1 - 2 \cdot 3^{-1/2} \cdot \cos\theta_1 + 3^{-1}}} = \frac{40\sqrt{5}}{297}  \\
    \beta_2' &= \beta_2 \cdot \frac{1}{\sqrt{1 - 2 \cdot 3^{-1/2} \cdot \cos\theta_2 + 3^{-1}}} = \frac{2\sqrt{19}}{27}
\end{align*}
Now we can use Corollary~\ref{cor:mainlambda} to compute the densities.
The area can be computed by writing them as single-variable integrals.
\begin{align*}
    \delta_+^{\lambda, \nc}(m) &= \frac{1}{2} + \frac{1}{2} \dd \phi \left(\left\{\frac{260}{891} > \frac{40\sqrt{5}}{297} \sin (x) + \frac{2\sqrt{38}}{27} \sin (y) > \frac{52\sqrt{3}}{891}\right\}\right) \approx 0.584867, \\
    \delta_+^{\lambda}(m) &= \frac{1}{2} + \frac{1}{2} \dd \phi \left(\left\{\frac{104}{297} > \frac{40\sqrt{5}}{297} \sin (x) + \frac{2\sqrt{19}}{27} \sin (y) > -\frac{52\sqrt{3}}{891}\right\}\right) \approx 0.739345.
\end{align*}
As in the previous example, we can compute \(b_n\) efficiently using the recurrence relation
\[
b_n = -b_{n-1} - b_{n-2} + 3b_{n-4} + 9b_{n-5} + 27b_{n-6}
\]
with initial terms \((b_0, b_1, b_2, b_3, b_4, b_5, b_6) = (1, -1, 0, 1, 1, 4, 12)\).
Table~\ref{tab:deg5} shows the approximations of \(\delta_+^{\lambda, \nc}(m)\) and \(\delta_+^\lambda(m)\) that converge to the actual densities as \(n\) increases.

\begin{table}
    \centering
    \begin{tabular}{c|cccc}
    \toprule
    \(n\) & \(\delta_+^{\lambda, \nc}(m)\) & \(\delta_+^{\lambda}(m)\) & \(\delta_+^{\mu, \nc}(m)\) & \(\delta_+^{\mu}(m)\) \\ 
    \midrule
    10    & 0.7000 & 0.6000 & 0.5000 & 0.5000 \\
    100   & 0.5500 & 0.7100 & 0.5500 & 0.4800 \\
    1000  & 0.5820 & 0.7360 & 0.5000 & 0.5010 \\
    10000 & 0.5849 & 0.7382 & 0.5005 & 0.4991 \\
    \bottomrule
    \end{tabular}
    \caption{Numerical computations of \(\lambda\) and \(\mu\)-densities for \(m = (T^2 + 1)(T^3 + 2T + 1) \in \bF_3[T]\) using polynomials of degree up to \(10, 100, 1000, 10000\).}
    \label{tab:deg5}
\end{table}